\documentclass[11pt,reqno,a4paper]{amsart}
\usepackage{amsmath,amssymb,amsthm} 
\usepackage{amsaddr}
\usepackage{mathtools}
\usepackage{url}
\usepackage{enumitem}
\usepackage[utf8]{inputenc}
\usepackage{float}
\usepackage{color}
\usepackage[british]{babel}
\usepackage[margin=3cm,top=4cm,bottom=4cm,footskip=1cm,headsep=2cm]{geometry}
\def\softd{{\leavevmode\setbox1=\hbox{d}%
		\hbox to 1.05\wd1{d\kern-0.4ex{\char039}\hss}}}

\title[Robust a posteriori error control for the Allen-Cahn equation]{Robust a posteriori error control for the Allen-Cahn equation with variable mobility }

\author{A. Brunk$^1$, J. Giesselmann$^2$ and M. Luk\'a\v{c}ov\'a-Medvi\softd ov\'a$^1$} 
\address{$^1$Institute of Mathematics, Johannes Gutenberg-University Mainz, Germany}
\address{$^2$Department of Mathematics, Technical University of Darmstadt, Germany}

\email{abrunk@uni-mainz.de} \email{giesselmann@mathematik.tu-darmstadt.de}
\email{lukacova@uni-mainz.de}

\def\NN{\mathbb{N}}
\def\RR{\mathbb{R}}
\def\WW{\mathbb{W}}
\def\QQ{\mathbb{Q}}

\def\Th{\mathcal{T}_h}
\def\Itau{\mathcal{I}_\tau}

\def\la{\langle}
\def\ra{\rangle}

\DeclarePairedDelimiter{\norm}{\|}{\|}
\DeclarePairedDelimiter{\snorm}{|}{|}
\newcommand{\na}{\nabla}
\def\E{E}
\def\G{\mathcal{G}}
\def\D{D}

\newtheorem{lemma}{Lemma}
\newtheorem{experiment}[lemma]{Experiment}
\newtheorem{theorem}[lemma]{Theorem}
\newtheorem{corollary}[lemma]{Corollary}
\theoremstyle{definition}
\newtheorem{remark}[lemma]{Remark}
\newtheorem{definition}[lemma]{Definition}

\newtheoremstyle{colon}%
{}
{}
{\itshape}
{}
{\bfseries}
{}
{ }
{}
\theoremstyle{colon}
\newtheorem*{problem*}{\phantom{A}}

\def\dt{\partial_t}

\def\Vh{\;\mathcal{V}_h}
\def\dx{\;\mathrm{d}x}
\def\dintt{\;\mathrm{d}t}
\def\ds{\;\mathrm{d}s}
\def\dtime{\frac{\mathrm{d}}{\mathrm{d}t}}

\usepackage{xcolor}

\begin{document}
\begingroup
\def\uppercasenonmath#1{} 
\let\MakeUppercase\relax 
\maketitle
\endgroup

\begin{abstract}
In this work, we derive a $\gamma$-robust a posteriori error estimator for finite element approximations of the Allen-Cahn equation with variable non-degenerate mobility. The estimator utilizes spectral estimates for the linearized steady part of the differential operator as well as a conditional stability estimate based on a weighted sum of Bregman distances, based on the energy and a functional related to the mobility. A suitable reconstruction of the numerical solution in the stability estimate leads to a fully computable estimator.
\end{abstract}

\begin{quote} 
\end{quote}
\begin{quote}
{\small { AMS-classification (2020):}
65M15, 
65M60, 
35K55 
35B35 
93D09 
}
\end{quote}

\vspace*{1em}

\section{Introduction}
Phase field models are often used to describe phase separation processes in binary mixtures, e.g., alloys \cite{ALLEN1979}. A popular phase field model is the Allen-Cahn equation (with variable mobility) 
\begin{align}
    \dt \phi &= b(\phi)\mu, \label{eq:s1}\\
    \mu &= -\Delta \phi + \frac{1}{\gamma}f'(\phi). \label{eq:s2}
\end{align}
Here $\phi$ is the phase field, $\mu$ is the chemical potential, $b=b(\phi)$ is the mobility and $f=f(\phi)$ is the energy density.
The physical free energy is given by
\begin{align}
E(\phi) = \int_\Omega \frac{1}{2}\snorm{\nabla\phi}^2 + \frac{1}{\gamma}f(\phi)\dx, \label{eq:energy}
\end{align}
where $\Omega\subset\RR^d, d\in\{1,2,3\}$ is the computational domain. The Allen-Cahn equation can then be derived as a (weighted) gradient flow of the above free energy functional in $L^2(\Omega)$, cf. \cite{Heida2015}. 

With regard to error estimates for numerical methods, the mathematical literature has focused on the system with constant mobility \cite{Shen2010} but from a modelling viewpoint the case of non-constant mobility is equally relevant \cite{Heida2015, GLW2017}.
The energy density $f$ is assumed to have a double-well shape giving rise to multi-phase effects.
The precise form of the double well $f$ 
 and the parameter $\gamma>0$ determines the width of the phase transition layer. In practical applications this layer is small compared to the computational domain, i.e. $0 < \gamma \ll 1$.  Thin phase transition layers make adaptive numerical schemes attractive and it is convenient to base them on a posteriori error estimates. To make such estimates applicable, it is important that they are robust with respect to small values of $\gamma$.

 It is well-known that standard error estimates based on 'naive' energy estimates for the Allen-Cahn equation depend exponentially on $\gamma^{-1}$, cf. \cite{Bartelsbook,Shen2016,Barrett2002,Li2019,Xiao2022}. For the Allen-Cahn equation with constant mobility it is by now classical that a priori and conditional a posteriori error estimates (for continuous finite element schemes) can be obtained that behave like low order polynomials in $\gamma^{-1}$, see \cite{FP2003,Yang2009,Akrivis2020} for a priori error estimates and \cite{Kessler2004,BMO2011,Bartels2011a,Bartels2011b,Bartels2015,Feng2005,Chry2020} for a posteriori error estimates. They are based on spectral estimates for the linearized spatial part of the Allen-Cahn operator, cf. \cite{Schatzman1995,Chen1994}. Error estimates for continuous finite elements in the $L^\infty(0,T,L^2(\Omega))$-norm have been extended to estimates in the $L^\infty(0,T,L^r(\Omega))$ norm with $r \in [2,\infty)$, see \cite{GM2014}; to recovery based error estimators \cite{Huang2019,CHY2020}, and to  stochastic Allen-Cahn equations \cite{BK2019}.

 The purpose of this paper is to extend the above results to the Allen-Cahn equation with non-degenerate non-constant mobility. It turns out that the non-constant mobility makes the problem considerably more non-linear. Thus, direct use of the $L^2$-norm in space for the stability estimates turns out to be inconvenient and it is beneficial to replace it with a Bregman distance with respect to a convex function induced by the mobility. We refer to \cite{EG1996} where this convex function has already been used to prove the existence of solutions to a Cahn-Hilliard equation with degenerate mobility. Using these stability estimates, which should be useful for the a posteriori analysis of various numerical methods, and the method of elliptic reconstruction \cite{Makridakis09}, we obtain a completely computable conditional a posteriori error estimate for a continuous finite element scheme. We observe polynomial scaling in $\gamma^{-1}$, the precise scaling is part of further research.
 Our estimates are analogous to the results for constant mobility in \cite{Bartels2015} with a slightly different eigenvalue.

 The remainder of this paper is structured as follows: In Section \ref{sec:ws} we recall basic definitions. Stability estimates are derived in Section \ref{sec:stab}. We describe our numerical method and suitable reconstructions of the numerical solution in Section \ref{sec:nm}. Finally, we experimentally assess the scaling behaviour of our error estimator with respect to $\gamma$ and discretisation parameters in Section \ref{sec:sim}.

\section{Weak solutions}\label{sec:ws}
In this section, we recall the definition of weak solutions for the Allen-Cahn equation, as well as their existence.

We will make the following assumptions
\begin{itemize}
    \item[(A0)] Let $\Omega\subset\mathbb{R}^d$ be a sufficiently regular domain, for $d\in\{1,2,3\}$;
    \item[(A1)] The interface parameter $\gamma>0$ is a positive constant;
    \item[(A2)] $b \in C^2(\RR , \RR_+)$  with $0< b_1\leq b(s) \leq b_2$, $\norm{b'}_\infty\leq b_3$, $\norm{b''}_\infty\leq b_4$ for constants $b_1,\ldots,b_4$;
    \item[(A3)] $f\in C^4(\mathbb{R})$ such that $f(s),f''(s)\geq -f_1$, for some $f_1\geq0$. Furthermore, we assume 
    that $f$ and its derivatives are bounded by $|f^{(k)}(s)| \le f_2^{(k)} + f_3^{(k)} |s|^{4-k}$ for $0 \le k \le 4$ with constants $f_2^{(k)}, f_3^{(k)}\geq 0$.
    \item[(A4)] The minima of $f(s)$ are located at $s=-1$ and $s=1$.
\end{itemize}
The usual double-well potential satisfying the above assumptions is given by $f(\phi)=\frac{1}{4}(\phi^2-1)^2$. The whole analysis can be done in a similar fashion for other potentials and other minima.

We denote the $L^2$-scalar product on $\Omega$ by $\la \cdot, \cdot \ra$
and the $W^{k,p}(\Omega)$ norm by $\| \cdot \|_{k,p}$. In particular, $\| \cdot \|_{0,p}$ is the $L^p(\Omega)$ norm. In the special case of $p=2$ we denote the $H^k(\Omega)=W^{k,2}(\Omega)$ norm by $\norm{\cdot}_k$ for $k\geq0$. Finally we denote the negative Sobolev space $H^{-1}(\Omega):=(H^1(\Omega))^{'}$ and its norm is given via $\norm{\cdot}_{-1}:=\sup_{0\neq v\in H^1(\Omega)}\frac{\la\cdot,v\ra}{\norm{v}_1}$. 

\begin{definition}\label{def:weaksol}
A pair $(\phi,\mu)$ with $\phi(t,x)\in[-1,1]$ for a.e. $(t,x)\in [0,T]\times\Omega$ and
\begin{align*}
\phi&\in L^\infty(0,T;H^1(\Omega))\cap L^2(0,T;H^2(\Omega))\cap H^1(0,T;L^2(\Omega)) =:\WW\\
\mu&\in L^2(0,T;L^2(\Omega)) =:\QQ
\end{align*}
is called a weak solution of the Allen-Cahn  problem on $[0,T]$, if
\begin{align}
\la \dt\phi,\psi \ra + \la b(\phi)\mu,\psi \ra &= 0, \label{eq:weak1} \\
\la \mu,\xi \ra - \la \na\phi,\na\xi\ra - \frac{1}{\gamma}\la f'(\phi),\xi \ra &= 0, \label{eq:weak2}
\end{align}
holds for all $(\psi,\xi)\in L^2(\Omega)\times H^1(\Omega)$ and a.e. $t \in [0,T]$.
\end{definition}
Given initial data $\phi_0\in H^1(\Omega)$
existence of weak solutions and uniqueness can be proven by Galerkin approximation; see \cite[Theorem 6.1]{Bartels2015}. Those solutions obey a maximum principle according to the minima of $f(\cdot)$ given suitable initial data. Note that the enhanced space regularity follows directly from elliptic regularity given the standard energy estimates. We note that weaker concepts of solutions for initial data $\phi_0\in L^2(\Omega)$ and their well-posedness are also available, cf. \cite[Theorem 6.1]{Bartels2015}.
 Weak solutions satisfy an energy inequality as stated in the following lemma.
\begin{lemma}
    Let $(\phi,\mu)$ be a weak solution of \eqref{eq:s1}--\eqref{eq:s2} in the sense of Definition \ref{def:weaksol}. Then the energy defined in \eqref{eq:energy} satisfies
    \begin{equation}
       \dtime E(\phi(t)) \leq -\int_\Omega b(\phi(t)) \mu(t)^2 \dx.  \label{eq:energinequ}
    \end{equation}
\end{lemma}
\begin{proof}
This follows from the energy inequality of Galerkin approximations and going to the limit using strong convergence and weak-lower semi-continuity of norms.
\end{proof}

\section{Stability}\label{sec:stab}
The purpose of this section is to provide $\gamma$-robust estimates for the difference between 
a weak solution $(\phi,\mu)$ of the Allen-Cahn equation and a strong solution 
\begin{align}
 (\hat \phi,\hat \mu) \in \hat\WW\times\hat\QQ :=     W^{1,\infty}(0,T; H^2(\Omega))\times 
   C^0(0,T; L^\infty(\Omega)).
\end{align}
The latter satisfies the following perturbed system for
given functions $r_1,r_2 \in L^2(0,T;L^2(\Omega))$:
\begin{align}
\la \dt\hat\phi,\psi \ra + \la b(\hat\phi)\hat\mu,\psi \ra &= \la r_1,\psi \ra, \label{eq:weakp1} \\
\la \hat\mu,\xi \ra - \la \na\hat\phi,\na\xi\ra - \frac{1}{\gamma}\la f'(\hat\phi),\xi \ra &= \la r_2,\xi \ra, \label{eq:weakp2}
\end{align}
 for all $(\psi,\xi)\in L^2(\Omega)\times H^1(\Omega)$ and a.e. $t \in [0,T]$.

In what follows we derive stability estimates that will be used to obtain a posteriori error estimates. Consequently, $\hat \phi, \hat \mu$ will be reconstructions of numerical solutions and $r_1, r_2$ the residuals that are obtained upon inserting the reconstructions into the continuous problem.

\subsection{Relative energy}
Let us begin by deriving an evolution inequality for the
 relative energy between $\phi$ and $\hat \phi$
\begin{align}
\E(\phi|\hat\phi) := \frac{1}{2}\norm{\nabla(\phi-\hat\phi)}_0^2 + \int_\Omega \frac{1}{\gamma}f(\phi|\hat\phi) \dx,  
\end{align}
\begin{equation}
   f(\phi|\hat\phi):= f(\phi) - f(\hat\phi) - f'(\hat\phi)(\phi-\hat\phi),
\end{equation}
and the 
 relative dissipation 
\begin{equation*}
\D_{\phi}(\mu|\hat\mu) := \norm{b^{1/2}(\phi)(\mu-\hat\mu)}_0^2.    
\end{equation*}

Let us recall that there is a continuous embedding $H^1(\Omega) \rightarrow L^6(\Omega), \Omega\subset\RR^d, d\leq 3.$ We denote its Lipschitz constant by $C_i$.

The relative energy satisfies
\begin{lemma}\label{lem:evolE}
    Let $(\phi,\mu)\in \WW\times\QQ$ be a weak solution of \eqref{eq:s1}--\eqref{eq:s2} in the sense of Definition \ref{def:weaksol}. Let $(\hat\phi,\hat\mu)\in \hat\WW\times\hat\QQ$ and $r_1, r_2$ satisfy \eqref{eq:weakp1}--\eqref{eq:weakp2}. Then the following estimate holds
\begin{align}\label{eq:evolE}
\dtime\E(\phi|\hat\phi) + \frac12 \D_\phi(\mu|\hat\mu) \leq &
\left( \frac{C_1}{\gamma} \|\dt \hat \phi\|_{0,\infty}  + C_4 \norm{\hat\mu}_{0,\infty}^2\right)  \| \phi - \hat \phi\|_{0}^2  \\  
& \quad + C_2\norm{r_1}_0^2 + C_3\norm{r_2}_0^2  \notag 
\end{align}
with
\begin{align}
C_1 &= 
\left( f_2^{(3)} + f_3^{(3)} (1 + \|\hat \phi\|_{0,\infty}) \right),  \label{eq:c1}\\ 
C_2 &= \frac12+\frac{3}{b_1},\qquad C_3 = 1 + 3b_2,
\qquad
C_4=  b_3 + \frac{6b_3^2}{b_1^2}. \label{eq:c24}
\end{align}
\end{lemma}

\begin{proof}
Computing the time evolution yields
\begin{align*}
\dtime\E(\phi|\hat\phi) &= \la \nabla(\phi-\hat\phi),\nabla\dt(\phi-\hat\phi) \ra + \frac{1}{\gamma}\la f'(\phi)-f'(\hat\phi),\dt(\phi-\hat\phi) \ra \\
&+ \frac{1}{\gamma}\la f'(\phi)-f'(\hat\phi)-f''(\hat\phi)(\phi-\hat\phi),\dt\hat\phi\ra  
\\
&=: (i) + (ii) + (iii) 
\end{align*}

In order to estimate the first two terms, we formally insert $\xi=\dt(\phi-\hat\phi)$
into the difference of \eqref{eq:weak2} and \eqref{eq:weakp2}.
This can be made rigorous by working with an approximation for $\phi$ exactly as in the proof of the energy (in)equality.
Afterwards we insert $\psi=\mu-\hat\mu+r_2$ into the difference of \eqref{eq:weak1} and \eqref{eq:weakp1} so that we obtain
\begin{align*}
&(i) + (ii) = \la \mu-\hat\mu + r_2,\dt(\phi-\hat\phi) \ra  \\
&= - \la b(\phi)(\mu-\hat\mu) , \mu-\hat\mu+r_2\ra + \la r_1,\mu-\hat\mu+r_2  \ra - \la (b(\phi)-b(\hat\phi))\hat\mu,\mu-\hat\mu+r_2 \ra\\
& \leq -(1-3\delta)\D_\phi(\mu|\hat\mu) + \left(\frac{1}{2}+\frac{1}{2\delta b_1}\right)\norm{r_1}_0^2+ \left(\frac{b_2}{2\delta}+1\right)\norm{r_2}_0^2\\
& \quad+ \left( b_3 + \frac{b_3^2}{\delta b_1^2}\right)\norm{\hat\mu}_{0,\infty}^2\norm{\phi-\hat\phi}_{0}^2
\end{align*}
for any $\delta>0$.

From the bounds in assumption (A3), we can deduce that 
\begin{align*}
|f'(\phi) - f'(\hat \phi) - f''(\hat \phi)(\phi - \hat \phi)| 
&\le \left(f_2^{(3)} + f_3^{(3)}(|\phi| + |\hat \phi|) \right) |\phi - \hat \phi|^2.
\end{align*}
Using H\"older's inequality and the maximum principle for $\phi$, we can  bound the third term in the above estimate by 
\begin{align*}
(iii) &\le \frac{1}{\gamma}\|\dt \hat \phi\|_{0,\infty} \|f'(\phi) - f'(\hat \phi) - f''(\hat \phi)(\phi - \hat \phi)\|_{0,1} \\
&\le \frac{1}{\gamma}\|\dt \hat \phi\|_{0,\infty} \left( f_2^{(3)} + f_3^{(3)} (\|\phi\|_{0,\infty} + \|\hat \phi\|_{0,\infty}) \right) \|\phi - \hat \phi\|_{0}^2 
\end{align*}
We obtain the assertion of the lemma by setting $\delta = \tfrac16$.

\end{proof}

\subsection{Robust stability estimate}
To obtain a robust stability estimate, we need to combine the relative energy investigated above with another relative functional, i.e. another Bregman distance (divergence), based on the mobility function. To this end, we introduce the function $\G:\RR\to \RR$ via
\begin{align}
   \G''(s) = b(s)^{-1},\ \G'(0)=0,\ \G(0)=0.
\end{align}

Since $b(\cdot)$ is strictly positive the function $\G(\cdot)$ is strictly convex. This type of function appears in the existence study of degenerated Cahn-Hilliard equations and is sometimes referred to as entropy, cf. \cite{EG1996}.
We consider the Bregman distance based on this function $\G(\phi)$, i.e., 
\begin{align*}
\G(\phi|\hat\phi):=  \G(\phi)-\G(\hat\phi) - \G'(\hat\phi)(\phi-\hat\phi) , \qquad G(\phi|\hat\phi):=\int_\Omega \G(\phi|\hat\phi)\dx  
\end{align*}
and aim to control the time evolution of $G(\phi|\hat \phi)$.

The principal eigenvalue of the steady part of the Allen-Cahn equation linearised around some given function $\bar\phi \in L^\infty(\Omega)$, i.e. $-\Delta\psi+\gamma^{-1}f''(\bar\phi)\psi=-\lambda_{\text{BMO}}\psi,$ has already been exploited in works considering $b\equiv 1$, see \cite{BMO2011}. In the mentioned work the authors use the fact that the principal eigenvalue can be characterized as the solution of the following minimization problem: 
\begin{align}
-\lambda_{\text{BMO}}(\gamma,\bar \phi) = \inf_{\psi\in H^1(\Omega)\setminus 
\{0\}}\frac{\norm{\nabla\psi}_0^2 + \tfrac{1}{\gamma}\la f''(\bar \phi),\psi^2\ra}{\norm{\psi}_0^2}.   \label{eq:eigenvalueBMO}
\end{align}

We consider a slightly different eigenvalue problem $-\tfrac{1}{2}\Delta\psi+\gamma^{-1}f''(\bar\phi)\psi=-\lambda\psi.$ And, hence, a slightly different minimization problem:
For some given $\bar\phi \in L^\infty(\Omega)$ let $\lambda$ solve
\begin{align}
-\lambda(\gamma,\bar \phi) = \inf_{\psi\in H^1(\Omega)\setminus 
\{0\}}\frac{\tfrac{1}{2}\norm{\nabla\psi}_0^2 + \tfrac{1}{\gamma}\la f''(\bar \phi),\psi^2\ra}{\norm{\psi}_0^2}.   \label{eq:eigenvalue}
\end{align}

\begin{remark}\label{rem:eig}
 For conditional estimates to be applicable it is essential to avoid an exponential dependence of the stability estimate on $\gamma^{-1}$.
 For the eigenvalue $\lambda_{\text{BMO}}$, using as $\bar\phi$ the solution of the Allen-Cahn equation, one expects that $\int_0^t \lambda_{+,\text{BMO}} \approx C + \log(\gamma^{-k})$ holds for suitable $k>0$  \cite{BMO2011}, where $\lambda_{+,\text{BMO}}$ denotes the positive part of $\lambda_{\text{BMO}}$. 
Hence, if we use a different eigenvalue it is essential to verify that this does not behave like a polynomial in $\gamma^{-1}$. The precise exponent $k$ is not essential for the stability estimate.
 Simple computations reveal that $2\lambda(2\gamma,\bar\phi)=\lambda_{\text{BMO}}(\gamma,\bar\phi)$. This means that for the same $\bar\phi$ a similar scaling law is expected. However, in the error estimates, $\bar\phi$ will be chosen as the numerical solution with the parameter $\gamma$ and not $2\gamma$. Hence, it is a priori unclear if the eigenvalue corresponding to another solution still follows a similar scaling law.
\end{remark}

We note that a similar analysis can also be performed using $\lambda_{\text{BMO}}$ and norms that scale differently in $\gamma$.

For $\gamma$ small enough we have $\int_\Omega \gamma f(\phi|\hat\phi) + \G(\phi|\hat\phi)\dx \geq \frac{b_1}{2} \| \phi - \hat \phi\|_{0}^2$
which implies 
$$\gamma^2 \E(\phi|\hat\phi) +G(\phi|\hat\phi)
\geq \frac{b_1}{2}\| \phi - \hat \phi\|_{0}^2  + \frac{\gamma^2}{2}\| \phi - \hat \phi\|_{1}^2. $$ For technical simplicity, we restrict ourselves to such $\gamma$ in the sequel.

The temporal evolution of $G(\phi|\hat\phi)$ is given in the following lemma.

\begin{lemma}\label{lem:evolG}
 Let $(\phi,\mu)\in \WW\times\QQ$ be a weak solution of \eqref{eq:s1}--\eqref{eq:s2} in the sense of Definition \ref{def:weaksol}. Let $(\hat\phi,\hat\mu)\in \hat\WW\times\hat\QQ$,  and $r_1, r_2$ satisfy \eqref{eq:weakp1}--\eqref{eq:weakp2}.
Let $\lambda(\gamma,\bar\phi(t))$ be the eigenvalue associated with the parameter $\gamma$ and a space-time function $\bar \phi$ with $\bar \phi \in L^\infty((0,T) \times \Omega)$ as specified in \eqref{eq:eigenvalue}. Then, the following estimate holds
 \begin{align}
&\dtime G(\phi|\hat\phi) + \frac{1}{16}\norm{\nabla(\phi-\hat\phi)}_0^2 \notag\\
\leq &
\left(\frac{1}{2} + \lambda(\gamma,\bar\phi) + \frac{b_3}{b_1^2}\norm{b(\hat\phi)\hat\mu}_{0,\infty}
\right)
\norm{\phi-\hat\phi}_0^2 \label{eq:temporalG}\\
&+  \frac{b_3 b_2}{b_1^2}\norm{\hat\mu}_{0,\infty}\norm{\phi-\hat\phi}_{0,3}^3 
 + \left(\frac{b_3 b_2 }{\gamma^2 b_1^2} +\frac{C_i^2}{\gamma^2} \right)\norm{\phi-\hat\phi}_{0,4}^4 
 \notag\\
&
+ \frac{\gamma^2}{4}D_\phi(\mu|\hat\mu)+ \norm{r_2}_{-1}^2 + \left(\frac{2}{b_1^2} + \frac{2b_3^2C_i^2}{b_1^4}\norm{\nabla\hat\phi}_{0,3}^2\right)\norm{r_1}_{-1}^2 
    + \frac{C_i^2}{\gamma^2 } \|f''(\hat \phi) - f''(\bar \phi)\|_{0}^4. \notag
\end{align}
 
\end{lemma}

\begin{proof}

We compute
\begin{align*}
&\dtime G(\phi|\hat\phi) \\
&= \la \dt\phi, \G'(\phi) \ra - \la \dt\hat\phi,G'(\hat\phi) \ra - \la \dt\hat\phi,\G''(\hat\phi)(\phi-\hat\phi) \ra - \la \dt(\phi-\hat\phi),\G'(\hat\phi) \ra \\
&= -\la \G'(\phi)-\G'(\hat\phi),b(\phi)\mu\ra + \la \G''(\hat\phi)b(\hat\phi)\hat\mu,(\phi-\hat\phi)\ra  - \la r_1,\G''(\hat\phi)(\phi-\hat\phi) \ra\\
&=-\la \G'(\phi)-\G'(\hat\phi)-\G''(\phi)(\phi-\hat\phi),b(\phi)\mu\ra  \\
& + \la \G''(\hat\phi)b(\hat\phi)\hat\mu-\G''(\phi)b(\phi)\mu,\phi-\hat\phi\ra- \la r_1,\G''(\hat\phi)(\phi-\hat\phi) \ra\\
&= \la \G'(\hat\phi|\phi),b(\phi)\mu \ra + \la \hat\mu-\mu,\phi-\hat\phi \ra - \la r_1,\G''(\hat\phi)(\phi-\hat\phi) \ra\\
&= (i) + (ii) + (iii).
\end{align*}

For the second term we test the difference of \eqref{eq:weak2} and \eqref{eq:weakp2} with $\phi - \hat \phi$ and obtain
\begin{align*}
(ii) &= -\norm{\nabla(\phi-\hat\phi)}_0^2 - \frac{1}{\gamma}\la f'(\phi)- f'(\hat\phi),\phi-\hat\phi \ra + \la r_2,\phi-\hat\phi\ra \\
&= -\norm{\nabla(\phi-\hat\phi)}_0^2 - \frac{1}{\gamma}\la f'(\phi|\hat\phi),\phi-\hat\phi \ra - \frac{1}{\gamma}\la f''(\hat\phi),(\phi-\hat\phi)^2\ra +\la r_2,\phi-\hat\phi\ra\\
&\leq -\frac{1}{4}\norm{\nabla(\phi-\hat\phi)}_0^2 + \lambda(\gamma,\bar\phi) \norm{\phi-\hat\phi}_0^2 - \frac{1}{\gamma}\la f'(\phi|\hat\phi),\phi-\hat\phi \ra + \norm{r_2}_{-1}^2 + \frac{1}{2}\norm{\phi-\hat\phi}_0^2\\
& -  \frac{1}{\gamma}\la f''(\hat \phi) - f''(\bar \phi),(\phi-\hat\phi)^2 \ra,
\end{align*}
where $\lambda(\gamma,\bar\phi)$ is the eigenvalue from \eqref{eq:eigenvalue}. We can control the last term as follows:

\begin{align*}
    \left|  \frac{1}{\gamma}\la f''(\hat \phi) - f''(\bar \phi),(\phi-\hat\phi)^2 \ra\right|
    &\leq
     \frac{1}{\gamma} 
     \| f''(\hat \phi) - f''(\bar \phi)\|_{0}
     \|\phi-\hat\phi\|_{0,3}  \|\phi-\hat\phi\|_{0,6}\\
    & \leq \frac {1}{16}  \|\phi-\hat\phi\|_{1}^2
     + \frac{2C_i^2}{\gamma^2} \| f''(\hat \phi) - f''(\bar \phi)\|_{0}^2
     \|\phi-\hat\phi\|_{0,3}^2\\
&\leq \frac {1}{16}  \|\phi-\hat\phi\|_{1}^2
     + \frac{C_i^2}{\gamma^2} \| f''(\hat \phi) - f''(\bar \phi)\|_{0}^4
     + \frac{C_i^2}{\gamma^2}
     \|\phi-\hat\phi\|_{0,3}^4.
\end{align*}

For the first term, we obtain
\begin{align*}
 (i) &= \la \G'(\hat\phi|\phi),b(\phi)(\mu-\hat\mu) + (b(\phi)-b(\hat\phi))\hat\mu + b(\hat\phi)\hat\mu\ra   \\
 &\leq \frac{\gamma^2}{4}\norm{b^{1/2}(\phi)(\mu-\hat\mu)}_0^2 + \frac{b_3 b_2 }{\gamma^2 b_1^2}\norm{\phi-\hat\phi}_{0,4}^4 \\
 &+ \frac{b_3 b_2}{b_1^2}\norm{\hat\mu}_{0,\infty}\norm{\phi-\hat\phi}_{0,3}^3 + \frac{b_3}{b_1^2} \norm{b(\hat\phi)\hat\mu}_{0,\infty}\norm{\phi-\hat\phi}_0^2. 
\end{align*}

Finally, we estimate the last term in the following way
\begin{align*}
(iii)=-  \la r_1,\G''(\hat\phi)(\phi-\hat\phi) \ra 
&\leq \norm{r_1}_{-1}\norm{\G''(\hat\phi)(\phi-\hat\phi)}_{1} \\
& \leq \norm{r_1}_{-1}\left(\frac{1}{b_1}\norm{\phi-\hat\phi}_{1} + \frac{b_3}{b_1^2}\norm{\phi-\hat\phi}_{0,6}\norm{\nabla\hat\phi}_{0,3}\right) \\
& \leq \left(\frac{2}{b_1^2} + \frac{2b_3^2C_i^2}{b_1^4}\norm{\nabla\hat\phi}_{0,3}^2\right)\norm{r_1}_{-1}^2 + \frac{1}{8}\norm{\phi-\hat\phi}_{1}^2.
\end{align*}

Here we used the continuous embedding of $H^1(\Omega)$ into $L^6(\Omega)$ and $C_i$ denoting the Lipschitz constant of this map. 
 Collecting the above estimates implies the assertion of the lemma.
\end{proof}

We will combine the assertions of Lemmas \ref{lem:evolE} and \ref{lem:evolG} in order to derive a robust conditional error estimate that is in the spirit of \cite[Theorem 3.1]{BMO2011}. It will hinge on the following generalized Gronwall lemma that extends the result in \cite[Lemma 2.1]{BMO2011}.
\begin{lemma}\label{lem:ggG}
Let $g_1, g_2, \alpha$ be non-negative functions such that $g_1\in C([0,T])$, $g_2\in L^1(0,T), \alpha\in L^\infty(0,T)$ and real numbers $A,B_1, B_2, \beta_1, \beta_2 > 0$ satisfy
\begin{multline*}
        g_1(t) + \int_0^t g_2(s) \ds \leq A + \int_0^t \alpha(s)g_1(s) \ds + \sum_{i=1}^2 B_i\sup_{\tilde s\in[0,T]} g_1(\tilde s)^{\beta_i}\int_0^t g_1(s) + g_2(s) \ds
\end{multline*}
for all $ t \in [0,T]$. Set $M=\exp(\int_0^T \alpha(s) \ds)$ and assume that 
\[ 8 (1+T) \left( B_1  (8AM)^{\beta_1} + B_2  (8AM)^{\beta_2}  \right) \leq 1 .  \]
 Then the following estimate holds
\begin{equation*}
    \sup_{t\in[0,T]} g_1(t) + \int_0^T g_2(s) \ds \leq 2AM.
\end{equation*}
\end{lemma}

\begin{proof}
    We set $\theta :=8AM$ and $\Upsilon(t):= \sup_{0\leq s \leq t} g_1(s) + \int_0^t g_2(s) \ds$ and define 
    \[I_\theta := \{ t \in [0,T] \; : \; \Upsilon(t) \leq \theta\}.
    \]
  Since $\Upsilon(0)=A \leq \theta$ and since $\Upsilon$ is continuous and monotone we have $I_\theta = [0,T_M]$ for some $0 < T_M \leq T$.
  We have, for every $t \in [0,T_M]$
  \begin{multline}
      \Upsilon(t) \leq A +  \int_0^t \alpha(s)g_1(s) \ds 
        +\sum_{i=1}^2 B_i\sup_{\tilde s\in[0,T]} g_1(\tilde s)^{\beta_i}\int_0^t g_1(s) + g_2(s) \ds\\
        \leq 
        A +  \int_0^t \alpha(s)\Upsilon(s) \ds + B_1 (1+T) \theta^{1+ \beta_1}
        +  B_2 (1+T) \theta^{1+ \beta_2}.
  \end{multline}
  Application of the classical Gronwall lemma implies for all $t \in [0,T_M]$
  \[ \Upsilon(t) \leq \left( A+B_1 (1+T) \theta^{1+\beta_1} + B_2 (1+T) \theta^{1+\beta_2}  \right)M   \leq \frac{\theta }{4},
  \]
  i.e. $\Upsilon(T_M) < \theta$ so that $T_M=T$.
\end{proof}

\begin{theorem}\label{thm:stab}
 Let $(\phi,\mu)\in \WW\times\QQ$ be a weak solution of \eqref{eq:s1}--\eqref{eq:s2} in the sense of Definition \ref{def:weaksol}. Let $(\hat\phi,\hat\mu)\in \hat\WW\times\hat\QQ$ and $r_1, r_2$ satisfy \eqref{eq:weakp1}--\eqref{eq:weakp2}. Let $\lambda(\gamma,\bar\phi)$ be the eigenvalue computed with respect to $\bar \phi\in L^\infty((0,T)\times\Omega)$ and $\gamma >0$.
Assume that
 \begin{equation}
      8 (1+T) \left( B_1  (8AM)^{\beta_1} + B_2  (8AM)^{\beta_2}  \right) \leq 1 , \label{eq:agg}
 \end{equation}
 for $\beta_1= 1/2$ and $ \beta_2 $ an arbitrary number smaller than $1$ for $d=1,2$ and $\beta_2= 2/3$ for $d=3$. Here
 and
 \begin{align*}
 A &:= \left. G(\phi|\hat\phi) + \gamma^2\E(\phi|\hat\phi)\right|_{t=0} + \int_0^t \gamma^2C_2\norm{r_1}_0^2 + \left(\frac{2}{b_1^2} + \frac{2b_3^2C_i^2}{b_1^4}\norm{\nabla\hat\phi}_{0,3}^2\right)\norm{r_1}_{-1}^2  \\
 & \qquad + C_3\gamma^2\norm{r_2}_0^2 + \frac{1}{2}\norm{r_2}_{-1}^2  +\frac{C_i^2}{\gamma^2 } \|f''(\hat \phi) - f''(\bar \phi)\|_{0}^4\ds, \\
\alpha &:= \frac{1}{2} + \lambda_+(\gamma,\bar\phi)  + \frac{ b_3}{b_1^2}\norm{b(\hat\phi)\hat\mu}_{0,\infty} 
+ C_1\gamma\|\partial_t \hat \phi\|_{0,\infty} + C_4\gamma^2\norm{\hat\mu}_{0,\infty}^2,\\
M& := \exp\left( \int_0^t \alpha(s) \ds \right), \quad
B_1 := 4 \max\{ 8, \frac{2}{b_1}\}  C_i^{3/2}
 \frac{b_3 b_2}{b_1^{5/2}}\norm{\hat\mu}_{0,\infty},\\
 B_2 &:=
  4 \max\{ 8, \frac{1}{b_1}\} \left(\frac{2}{b_1}\right)^{2\beta_2} C_e^2 \left( \frac{ b_3 b_2 }{\gamma^2 b_1^2}+
  \frac{C_i^2}{\gamma^2}\right) (1 + \| \hat \phi\|_{0,\infty})^{2-2\beta_2}.
\end{align*}
Then the following error estimate holds
\begin{align*}
    \sup_{t\in[0,T]} &\left(  G(\phi|\hat\phi)(t) + \gamma^2\E(\phi|\hat\phi)(t)\right)+ \int_0^T \frac{\gamma^2}{4}D_\phi(\mu|\hat\mu) + \frac{1}{16}\norm{\nabla(\phi-\hat\phi)}_0^2\ds \\
    &\leq 8M\Big( G(\phi|\hat\phi)(0) +  \gamma^2\E(\phi|\hat\phi)(0) + \int_0^T C_2\gamma^2\norm{r_1}_0^2 + \left(\frac{2}{b_1^2} + \frac{2b_3^2C_i^2}{b_1^4}\norm{\nabla\hat\phi}_{0,3}^2\right)\norm{r_1}_{-1}^2\ds \\
    &\qquad\qquad+ \int_0^T C_3\gamma^2\norm{r_2}_0^2  + \norm{r_2}_{-1}^2 + \frac{C_i^2}{\gamma^2 } \|f''(\hat \phi) - f''(\bar \phi)\|_{0}^4 \ds\Big).   
\end{align*}

\end{theorem}

\begin{proof}
We multiply the \eqref{eq:evolE} with $\gamma^2$ and add it to \eqref{eq:temporalG}. 
This gives the following result
\begin{align*}
\dtime &\Big(G(\phi|\hat\phi) + \gamma^2\E(\phi|\hat\phi)\Big) + \frac{\gamma^2}{2}D_\phi(\mu|\hat\mu) + \frac{1}{16}\norm{\nabla(\phi-\hat\phi)}_0^2 
\\ & \leq
\left(\frac{1}{2} + \lambda(\gamma,\bar\phi)  + \frac{ b_3}{b_1^2}\norm{b(\hat\phi)\hat\mu}_{0,\infty} 
+ \gamma^2 C_4 \norm{\hat\mu}_{0,\infty}^2 
+ C_1\gamma \| \partial_t \hat \phi\|_{0,\infty}
\right)
\norm{\phi-\hat\phi}_0^2\\
&+   \frac{b_3 b_2}{b_1^2}\norm{\hat\mu}_{0,\infty}\norm{\phi-\hat\phi}_{0,3}^3 
 +\left( \frac{ b_3 b_2 }{\gamma^2 b_1^2} + \frac{C_i^2}{\gamma^2}\right) \norm{\phi-\hat\phi}_{0,4}^4
 \\
&
+ \frac{\gamma^2}{4}D_\phi(\mu|\hat \mu)
+\norm{r_2}_{-1}^2 + \left(\frac{2}{b_1^2} + \frac{2b_3^2C_i^2}{b_1^4}\norm{\nabla\hat\phi}_{0,3}^2\right)\norm{r_1}_{-1}^2 
+ \frac{C_i^2}{\gamma^2 } \|f''(\hat \phi) - f''(\bar \phi)\|_{0}^4 \\
&+ C_2\gamma^2\norm{r_1}_0^2 + C_3\gamma^2\norm{r_2}_0^2.
\end{align*}

Further, we can control the higher-order terms as follows
\begin{align*}
\norm{\phi-\hat\phi}_{0,3}^3 &\leq C_i^{3/2}\norm{\phi-\hat\phi}_0(\norm{\phi-\hat\phi}_0^2 + \norm{\nabla(\phi-\hat\phi)}_0^2), \text{ and } \nonumber  \\ 
\norm{\phi-\hat\phi}_{0,4}^4
&\leq \norm{(\phi-\hat\phi)^{\tfrac{2}{p}}}_{0,p}
\norm{(\phi-\hat\phi)^{2}}_{0,q}
\norm{(\phi-\hat\phi)^{2-\tfrac{2}{p}}}_{0,\infty}
\\
& \leq 
C_e^2\norm{\phi-\hat\phi}_{0}^{2/p}\norm{\phi-\hat\phi}_{1}^{2}\norm{\phi-\hat\phi}_{0,\infty}^{2- \tfrac{2}{p}}
 \nonumber 
\end{align*}
where $0< \beta_2< 1$ is arbitrary for $d=1,2$ and $0< \beta_2< \tfrac23$ for $d=3$; $p=\tfrac{1}{\beta_2}$, and $q=  \tfrac{1}{1-\beta_2}$, and $C_e$ is the Lipschitz constant of the embedding from $H^1$ to $L^{2q}.$

The above estimates and integration in time leads to
\begin{align*}
& \left(G(\phi|\hat\phi) + \gamma^2\E(\phi|\hat\phi)\right)(t) + \int_0^t  \frac{\gamma^2}{4}D_\phi(\mu|\hat\mu) + \frac{1}{16}\norm{\nabla(\phi-\hat\phi)}_0^2 \ds
 \\
 &\quad \leq \left( G(\phi|\hat\phi) + \gamma^2\E(\phi|\hat\phi)\right)(0) +   \int_0^tC_2\gamma^2\norm{r_1}_0^2 + \left(\frac{2}{b_1^2} + \frac{2b_3^2C_i^2}{b_1^4}\norm{\nabla\hat\phi}_{0,3}^2\right)\norm{r_1}_{-1}^2\\
 &\qquad  + C_3\gamma^2\norm{r_2}_0^2 + \frac{1}{2}\norm{r_2}_{-1}^2 + \alpha(s)\norm{\phi-\hat\phi}_0^2 \; + \frac{C_i^2}{\gamma^2 } \|f''(\hat \phi) - f''(\bar \phi)\|_{0}^4 \ds\\
 &\qquad +  \frac{b_3 b_2}{b_1^2}\norm{\hat\mu}_{0,\infty} C_i^{3/2}\sup_{t\in[0,T]}\norm{\phi-\hat\phi}_0
 \int_0^t (\norm{\phi-\hat\phi}_0^2 + \norm{\nabla(\phi-\hat\phi)}_0^2) \ds\\
&\qquad+ C_e^2 \left( \frac{ b_3 b_2 }{\gamma^2 b_1^2}+
  \frac{C_i^2}{\gamma^2}\right) (1 + \| \hat \phi\|_{0,\infty})^{2-2\beta_2}\sup_{t\in[0,T]}\norm{\phi-\hat\phi}_0^{2\beta_2}
 \int_0^t (\norm{\phi-\hat\phi}_0^2 + \norm{\nabla(\phi-\hat\phi)}_0^2) \ds.
\end{align*}

Application of the generalized Gronwall lemma, Lemma \ref{lem:ggG}, with 
\begin{align}
g_1 &:=   G(\phi|\hat\phi) + \gamma^2\E(\phi|\hat\phi) , \label{eq:conda1}\\
g_2 &:= \frac{\gamma^2}{4}D_\phi(\mu|\hat\mu) + \frac{1}{16}\norm{\nabla(\phi-\hat\phi)}_0^2,
\end{align}
and assumption \eqref{eq:agg} finish the proof.
\end{proof}

In the special case of constant mobility, i.e. $b\equiv 1$, we can compute the Bregman distance and the related functional and find that $G(\phi|\hat\phi)=\frac{1}{2}\norm{\phi-\hat\phi}_0^2$. In this situation, the conditional stability estimates are simplified.
\begin{corollary}
    Let $(\phi,\mu)\in \WW\times\QQ$ be a weak solution of \eqref{eq:s1}--\eqref{eq:s2} in the sense of Definition \ref{def:weaksol}. Let $(\hat\phi,\hat\mu)\in \hat\WW\times\hat\QQ$ and $r_1, r_2$ satisfy \eqref{eq:weakp1}--\eqref{eq:weakp2}. Let $\lambda(\gamma,\bar\phi)$ be the eigenvalue computed with respect to $\bar \phi\in L^\infty((0,T)\times\Omega)$ and $\gamma >0$, cf. \eqref{eq:eigenvalue}.
 Assume
 \begin{equation}
      8 (1+T) B_2  (8AM)^{\beta_2}  \leq 1 ,
 \end{equation}
 for $ \beta_2 $ an arbitrary number smaller than $1$ for $d=1,2$ and $\beta_2= 2/3$ for $d=3$,
 and
 \begin{align*}
 A &:= \left. \frac{1}{2}\norm{\phi-\hat\phi}_0^2 + \gamma^2\E(\phi|\hat\phi)\right|_{t=0} + \int_0^t \frac{7}{2}\gamma^2\norm{r_1}_0^2 + 2\norm{r_1}_{-1}^2  \\
 & \qquad + 4\gamma^2\norm{r_2}_0^2 + \frac{1}{2}\norm{r_2}_{-1}^2  +\frac{C_i^2}{\gamma^2 } \|f''(\hat \phi) - f''(\bar \phi)\|_{0}^4\ds, \\
\alpha &:= \frac{1}{2} + \lambda_+(\gamma,\bar\phi) 
+ C_1\gamma\|\partial_t \hat \phi\|_{0,\infty},\\
M& := \exp\left( \int_0^t \alpha(s) \ds \right), \quad
B_2 :=
  32 \cdot 2^{2\beta_2} C_e^2 \left(
  \frac{C_i^2}{\gamma^2}\right) (1 + \| \hat \phi\|_{0,\infty})^{2-2\beta_2}.
\end{align*}
Then the following error estimate holds
\begin{align*}
    \sup_{t\in[0,T]} &\left(  \frac{1}{2}\norm{\phi(t)-\hat\phi(t)}_0^2 + \gamma^2\E(\phi|\hat\phi)(t)\right)+ \int_0^T \frac{\gamma^2}{4}D_\phi(\mu|\hat\mu) + \frac{1}{16}\norm{\nabla(\phi-\hat\phi)}_0^2\ds \\
    &\leq 8M\Big( \frac{1}{2}\norm{\phi(0)-\hat\phi(0)}_0^2 +  \gamma^2\E(\phi|\hat\phi)(0) \\
    &\qquad\quad+ \int_0^T \frac{7}{2}\gamma^2\norm{r_1}_0^2 + 2\norm{r_1}_{-1}^2  + 4\gamma^2\norm{r_2}_0^2  + \norm{r_2}_{-1}^2 + \frac{C_i^2}{\gamma^2 } \|f''(\hat \phi) - f''(\bar \phi)\|_{0}^4 \ds\Big).
\end{align*}
\end{corollary}

Note that in contrast to \cite{BMO2011}, where the residuals are estimated in the $L^2$-norm, we estimate the residuals in the $H^{-1}$-norm and a $\gamma^2$-weighted $L^2$-norm. This is mainly due to the use of the modified eigenvalue \eqref{eq:eigenvalue} that allows to derive a problem-adapted posteriori error estimate.

\begin{remark}
 In accordance with the literature we only take the positive part of $\lambda$, i.e. $\lambda_+$, in $\alpha$. However, taking directly $\alpha_+$ is also valid and, in fact, sharper.  
\end{remark}


\section{Numerical method and reconstructions}\label{sec:nm}
In this section, we will briefly present properties of the numerical method, for which we derive a posteriori estimates. We also introduce the reconstructions of the numerical solution that are needed to employ the stability theory from Section \ref{sec:stab}. Finally, we will discuss how to compute bounds on suitable norms of the residuals. We note that the choice of a particular numerical scheme is done for simplicity of presentation and 
 in order to illustrate how the conditional stability estimate is employed. The conditional stability estimate is however independent of the specific scheme and can be used for other numerical schemes as well. 

 \subsection{Numerical method}
 Let us decompose the time interval into $N$ equidistant time steps as $0= t^0 < t^1 < \dots < t^N=T$, i.e. there exists $\tau>0$ such that $t^j= j \cdot \tau$ for $j=0,\dots,N$. We introduce the time grid $\Itau:=\{t^0,\ldots,t^N\}$ and denote by $\mathbb{P}_1^c(\Itau), \mathbb{P}_0(\Itau)$ the space of continuous piecewise linear and (discontinuous) piecewise constant functions on $\Itau$.

 Let $\Th$ denote a conforming triangulation of $\Omega$ into triangles of maximal diameter $h=h_{\max}$ and minimal diameter $h_{\min}.$ For a fixed $p\in\NN$ we denote the set of polynomials in $d$ real variables of degree at most $p$ by $\mathbb{P}_p$ and we define the space of conforming finite element functions over $\Th$ via
 \begin{equation}
  \Vh := \{ v\in H^1(\Omega) \, : \, v|_K \in \mathbb{P}_p \ \forall \; K \in \Th\}.
 \end{equation}
The $L^2$-orthogonal projection from $L^2(\Omega)$ into $\Vh$ is denoted by $\pi_h$.

The discrete numerical method for the Allen-Cahn equation \eqref{eq:s1}--\eqref{eq:s2} is given as follows.
	\begin{problem*}\label{prob:semi}
	Given $\phi_h^0 = \pi_h \phi_0$, find $\phi_{h}^{n+1},\mu_{h}^{n+1}\in \Vh$ for $n=0, \dots, N$ such that 
		\begin{align}
			\frac{1}{\tau}\la \phi_{h}^{n+1}-\phi_{h}^{n},\psi_h \ra + \la b(\phi_{h}^n)\mu_{h}^{n+1},\psi_h \ra &= 0, \label{eq:disc1}\\
			\la \mu_{h}^{n+1},\xi_h \ra - \la \na\phi_{h}^{n+1},\na\xi_h\ra - \frac{1}{\gamma}\la f'(\phi_{h}^{n+1}),\xi_h \ra &= 0, \label{eq:disc2}
		\end{align}
		holds for all $\psi_h,\xi_h\in\Vh$ and for every $n\geq0$.
	\end{problem*}
	
 For simplicity, we focus on implicit discretisation, with an explicit evaluation of the mobility function $b$. However, the techniques below can directly be extended to an explicit method or a convex-concave splitting. 

It is well-known, cf. \cite{Shen2010}, that the implicit Euler discretisation is not unconditionally energy-stable, i.e. the discrete version of \eqref{eq:energinequ} holds only for suitable choices of $\tau$, here $\tau\leq O(\gamma)$. The existence of a discrete solution can be shown in standard fashion by realising that the numerical scheme is the first-order optimality condition of a convex functional.

\subsection{Reconstruction}
In order to apply the stability theory from Section \ref{sec:stab} we need functions that are more regular in space and time than our numerical solution.
In particular, we use elliptic reconstruction that was introduced in \cite{MN2003}.
First we introduce the linear interpolants $\phi_{h,\tau},\mu_{h,\tau}\in \mathbb{P}^c_1(I_\tau;\Vh)$ such that
	\begin{align*}
		\phi_{h,\tau}(t^k)=\phi_h^k, \qquad \mu_{h,\tau}(t^k)=\mu_h^k \text{ for all } k\geq 0,
	\end{align*}
which can be interpreted as a reconstruction in time. In order to obtain a meaningful linear interpolant for $\mu_{h,\tau}$ we fix the initial value \begin{equation}\label{eq:why}
	 \la\mu_{h,\tau}(0), \xi_h \ra =  \la\mu_h^0, \xi_h \ra := \la \nabla\phi_h^0,\nabla\xi_h \ra + \frac{1}{\gamma}\la f'(\phi_h^0),\xi_h\ra .
   \end{equation}

Our numerical method can then be rewritten as follows.
\begin{problem*}\label{prob:semirev}
		Find $(\phi_{h,\tau},\mu_{h,\tau})\in \mathbb{P}^c_1(I_\tau;\Vh\times\Vh)$ with $\phi_{h,\tau}(0)=\phi_h^0, \mu_{h,\tau}(0)=\mu_h^0$ such that 
		\begin{align}
			\la \dt\phi_{h,\tau},\psi_h \ra + \la b(\phi_{h,\tau}(t^n))\mu_{h,\tau}(t^{n+1}),\psi_h \ra &= 0, \label{eq:disc1b}\\
			\la \mu_{h,\tau},\xi_h \ra - \la \na\phi_{h,\tau},\na\xi_h\ra - \frac{1}{\gamma}\la I_1[f'(\phi_{h,\tau})],\xi_h \ra &= 0\label{eq:disc2b}
		\end{align}
	holds for all $\psi_h,\xi_h\in\Vh$ and for every $0\leq n \leq N-1$.
\end{problem*}
	The linear reconstruction for the nonlinear term using the standard Lagrange basis reads $I_1[f'(\phi_{h,\tau})] \in \mathbb{P}^c_1(I_\tau; L^2(\Omega))$
	\begin{align*}
		I_1[f'(\phi_{h,\tau})](t) &:= \frac{t-t^{n}}{\tau}f'(\phi_{h,\tau}(t^{n+1})) - \frac{t-t^{n+1}}{\tau}f'(\phi_{h,\tau}(t^{n})). 
\end{align*}

We now describe  reconstruction in space:
 For the chemical potential $\mu$ we use the reconstruction $\hat\mu=\mu_{h,\tau},$ while we apply 
  elliptic reconstruction for the phase fraction $\phi$. This means that $\hat\phi \in \mathbb{P}^c_1(I_\tau;H^1(\Omega))$ is defined as solution of the elliptic problem
	\begin{equation*}
		\la\na\hat\phi,\nabla v\ra = -\la \Delta_h\phi_{h,\tau}, v\ra = \la \mu_{h,\tau} -\frac{1}{\gamma}I_1[f'(\phi_{h,\tau})],v \ra\quad \forall v \in H^1(\Omega),
	\end{equation*}
 where $\Delta_h: \Vh \rightarrow \Vh$ is the discrete Laplacian defined by
 \[
 -   \la \Delta_h \xi_h , \zeta_h \ra 
 = \la \nabla \xi_h, \nabla \zeta_h \ra 
 \quad \forall \xi_h, \zeta_h \in \Vh.
 \]
 Note that actually $\hat\phi \in \mathbb{P}^1(I_\tau;H^2(\Omega))$ due to elliptic regularity.
	
	Using the above construction we obtain the following perturbed system
	\begin{align*}
		\la \dt\hat\phi,\psi \ra  + \la b(\hat\phi)\hat\mu,\psi \ra = \la r_1,\psi \ra, \\
		\la \hat\mu,\xi \ra  - \la \nabla\hat\phi,\nabla\xi \ra - \la f'(\hat\phi),\xi \ra = \la r_2,\xi \ra
	\end{align*}
	with the residuals
	\begin{align*}
		\la r_1,\psi \ra &=  \la \dt\hat\phi,\psi \ra  + \la b(\hat\phi)\hat\mu,\psi \ra \\
		&=\la \dt\hat\phi-\dt\phi_{h,\tau},\psi \ra  + \la b(\hat\phi)\mu_{h,\tau}-b(\phi_{h,\tau}(t^n)\mu_{h,\tau}(t^{n+1})),\psi \ra \\
		&= \la \dt\hat\phi-\dt\phi_{h,\tau},\psi \ra + \la(b(\hat\phi)-b(\phi_{h,\tau}))\mu_{h,\tau}(t^{n+1})),\psi \ra   \\
		& + \la (b(\phi_{h,\tau})-b(\phi_{h,\tau}(t^n)))\mu_{h,\tau}(t^{n+1}),\psi\ra+ \la b(\hat\phi)(\mu_{h,\tau}-\mu_{h,\tau}(t^{n+1})), \psi\ra, \\
		\la r_2,\xi \ra &= \la \mu_{h,\tau},\xi \ra  - \la \nabla\hat\phi,\nabla\xi \ra - \la f'(\hat\phi),\xi \ra \\
		&= \frac{1}{\gamma}\la I_1[f'(\phi_{h,\tau})]-f'(\hat\phi),\xi \ra  \\
		&= \frac{1}{\gamma}\la I_1[f'(\phi_{h,\tau})]-f'(\phi_{h,\tau}) + f'(\phi_{h,\tau})-f'(\hat\phi),\xi \ra.
	\end{align*}

\subsection{Computable bounds for residuals}
Since the reconstructions involve exact solutions of elliptic problems they cannot be computed explicitly and the same holds for the (norms of) residuals.

Following \cite{Makridakis09} we introduce the following error indicator
\begin{align*}
\eta_{p,-j}(g,v_h,K)&:=h_K^{2+j}\norm{g+\Delta v_h}_{0,p,K} + h_K^{j+1+\frac{1}{p}}\norm{[[\nabla v_h]]}_{0,p,\partial K}, \\
\mathcal{E}_{p,-j}(v_h,g)&:=\begin{cases}
    \left( \sum\limits_{K\in\Th}\eta_{p,-j}(g,v_h,K)^p\right)^{1/p} \text{ for } 1\leq p < \infty,\\
    \max\limits_{K\in\Th} \eta_{\infty,-j}(g,v_h,K) \qquad\quad \text{ for } p=\infty.
\end{cases}
\end{align*}
Here  we write $\| \cdot\|_{0,p,K}$ and   $\| \cdot\|_{0,p,\partial K}$ to distinguish the $L^p$-norm on the triangles $K$ and the edges $\partial K$, $[[\cdot ]]$ denotes jumps across edges and $h_K$ denotes the diameter of cell $K$.

An important observation in elliptic reconstruction is that $L^2$-, $H^{-1}$-, and  $H^1$-norms  of $\hat\phi-\phi_{h,\tau}$ and its time derivative can be bounded by residual type a posteriori error estimators, see e.g. \cite{MN2003}. We denote these estimators by $H_{-1}, H_0, H_1$ and they are defined as follows
\begin{align*}
		\norm{\hat\phi-\phi_{h,\tau}}_0^2 &\leq H_0[\phi_{h,\tau}]:= \mathcal{E}_{2,0}\left(\phi_{h,\tau},\mu_{h,\tau}-\tfrac{1}{\gamma}I_1[f'(\phi_{h,\tau})]\right)^2, \\
		\norm{\hat\phi-\phi_{h,\tau}}_1^2 &\leq H_1[\phi_{h,\tau}]:= \mathcal{E}_{2,1}\left(\phi_{h,\tau},\mu_{h,\tau}-\tfrac{1}{\gamma}I_1[f'(\phi_{h,\tau})]\right)^2, \\
        \norm{\hat\phi-\phi_{h,\tau}}_{-1}^2 &\leq H_{-1}[\phi_{h,\tau}]:= \mathcal{E}_{2,-1}\left(\phi_{h,\tau},\mu_{h,\tau}-\tfrac{1}{\gamma}I_1[f'(\phi_{h,\tau})]\right)^2, \\
		\norm{\dt(\hat\phi-\phi_{h,\tau})}_0^2 &\leq H_0[\dt\phi_{h,\tau}]:= \mathcal{E}_{2,0}\left(\dt\phi_{h,\tau},\dt(\mu_{h,\tau}-\tfrac{1}{\gamma}I_1[f'(\phi_{h,\tau}))]\right)^2,\\
        \norm{\dt(\hat\phi-\phi_{h,\tau})}_{-1}^2 &\leq H_{-1}[\dt\phi_{h,\tau}]:= \mathcal{E}_{2,-1}\left(\dt\phi_{h,\tau},\dt(\mu_{h,\tau}-\tfrac{1}{\gamma}I_1[f'(\phi_{h,\tau}))]\right)^2.
	\end{align*}
	 Note that the estimator for the negative norm for the time derivative does only scale with the above mentioned powers of $h$ if we use finite elements with polynomial degree $p\geq 2$. We will assume this condition in the rest of this paper. In the case of the lowest order, i.e. piecewise linear polynomials, the estimator for the $H^{-1}$-norm scales with the same powers of $h$ as the $L^2$-norm estimator, cf. \cite{Makridakis09}.
	
	Furthermore, we compute the time errors
	\begin{align*}
		\phi_{h,\tau}-\phi_h^n = (t-t^n)\dt\phi_{h,\tau}  , \qquad \mu_{h,\tau}-\mu_h^{n+1} = (t-t^{n+1})\dt\mu_{h,\tau}.
	\end{align*}

Note that $\phi_{h,\tau}$ and $ \mu_{h,\tau}$ are piecewise linear in time and hence the time derivatives are piecewise constant, hence we will denote them as $\dt\phi_{h,\tau}^{n+1},\dt\mu_{h,\tau}^{n+1}$.

 In the next step the residuals $r_1,r_2$ are estimated in the $L^2$- and $H^{-1}$-norm, respectively. Recall that the upper bounds need to be computable from the numerical solution.
	Let us  estimate the first residual in $L^2$ and in $H^{-1}$, where $I_n := [t^n, t^{n+1}]$,
	\begin{align}
		\int_{I_n}\norm{r_1}_0^2 \;dt&\leq \int_{I_n} H_0[\dt\phi_{h,\tau}] + b_3^2H_1[\phi_{h,\tau}] \norm{\mu_{h}^{n+1}}_{0,3}^2 \label{eq:est1}\\
		&+ b_3^2(t-t^n)^2\norm{\dt\phi_{h,\tau}^{n+1}}_{0,6}^2\norm{\mu_{h}^{n+1}}_{0,3}^2 + b_2^2(t-t^{n+1})^2\norm{\dt\mu_{h,\tau}^{n+1}}_{0}^2 \dintt\notag\\
        & \leq \int_{I_n} H_0[\dt\phi_{h,\tau}] + b_3^2H_1[\phi_{h,\tau}] \norm{\mu_{h}^{n+1}}_{0,3}^2 \dintt \notag\\
        & + \frac{\tau^3}{3}\left(b_3^2\norm{\dt\phi_{h,\tau}^{n+1}}_{0,6}^2\norm{\mu_{h}^{n+1}}_{0,3}^2 + b_2^2\norm{\dt\mu_{h,\tau}^{n+1}}_{0}^2\right), 
 \notag\\
        \int_{I_n}\norm{r_1}_{-1}^2 \dintt &\leq \int_{I_n} H_{-1}[\dt\phi_{h,\tau}] + b_3^2H_0[\phi_{h,\tau}] \norm{\mu_{h}^{n+1}}_{0,3}^2 \label{eq:est2}\\
		&+ b_3^2(t-t^n)^2\norm{\dt\phi_{h,\tau}^{n+1}}_{0}^2\norm{\mu_{h}^{n+1}}_{0,3}^2 + (t-t^{n+1})^2(b_2^2 + b_3^2\norm{\nabla\hat\phi}_{0,3}^2)\norm{\dt\mu_{h,\tau}^{n+1}}_{-1}^2  \dintt \notag\\
        &\leq \int_{I_n} H_{-1}[\dt\phi_{h,\tau}] + b_3^2H_0[\phi_{h,\tau}] \norm{\mu_{h}^{n+1}}_{0,3}^2 \dintt \notag\\
        &+ \norm{\dt\mu_{h,\tau}^{n+1}}_{-1}^2\int_{I_n}(t-t^{n+1})^2(b_3^2\norm{\nabla(\phi_{h,\tau}-\hat\phi)}_{0,3}^2 + b_3^2\norm{\nabla\phi_{h,\tau}}_{0,3}^2) \dintt\notag\\
  &+ \tau^3\left(b_3^2\norm{\dt\phi_{h,\tau}^{n+1}}_{0}^2\norm{\mu_{h}^{n+1}}_{0,3}^2+b_2^2\norm{\dt\mu_{h,\tau}^{n+1}}_{-1}^2\right).  \notag
	\end{align}
	For the second residual we find
	\begin{align}
		\gamma^2\int_{I_n}\norm{r_2}_0^2 \dintt&\leq \int_{I_n} \norm{I_1[f'(\phi_{h,\tau})]-f'(\phi_{h,\tau})}_0^2 + \norm{f'(\phi_{h,\tau})-f'(\hat\phi)}_0^2 \dintt \label{eq:est3}\\
  &\leq \int_{I_n} \norm{I_1[f'(\phi_{h,\tau})]-f'(\phi_{h,\tau})}_0^2 +  (f_2^{(2)}+ f_3^{(2)}(\norm{\hat\phi}_{0,\infty}^2+ \norm{\phi_{h,\tau}}_{0,\infty}^2))^2H_0[\phi_{h,\tau}]\dintt, \notag\\
  \gamma^2\int_{I_n}\norm{r_2}_{-1}^2 \dintt&\leq \int_{I_n} \norm{I_1[f'(\phi_{h,\tau})]-f'(\phi_{h,\tau})}_{-1}^2  \label{eq:est4}\\
  & + (f_2^{(2)}+ f_3^{(2)}(\norm{\hat\phi}_{0,\infty}^2+ \norm{\phi_{h,\tau}}_{0,\infty}^2))^2H_{-1}[\phi_{h,\tau}]  \notag\\
  & +(f_2^{(3)}+ f_3^{(3)}(\norm{\hat\phi}_{0,\infty}+ \norm{\phi_{h,\tau}}_{0,\infty}))^2(\norm{\nabla\phi_{h,\tau}}_{0,3}+\norm{\nabla\hat\phi}_{0,3})^2H_{-1}[\phi_{h,\tau}]\dintt. \notag
	\end{align}
The norm $\norm{\hat\phi}_{0,\infty}$ is estimated as follows
\begin{align}
 \norm{\hat\phi}_{0,\infty} &\leq \ln(h_{\min})^2
 \mathcal{E}_{\infty,0}\left(\phi_{h,\tau},\mu_{h,\tau}-\tfrac{1}{\gamma}I_1[f'(\phi_{h,\tau})]\right) +\norm{\phi_{h,\tau}}_{0,\infty}  \label{eq:resinf}
\end{align}
where we have employed a residual estimator for the $L^\infty$-norm as in \cite{Makridakis09}. Similarly we employ a residual estimator for the $L^3$ norm of $\nabla(\phi_{h,\tau}-\hat\phi)$ following \cite{Makridakis09,Araya2006} given by
\begin{align*}
 \norm{\phi_{h,\tau}-\hat\phi}_{1,3}^2 \leq \mathcal{E}_{3,1}\left(\phi_{h,\tau},\mu_{h,\tau}-\tfrac{1}{\gamma}I_1[f'(\phi_{h,\tau})]\right)^2.
\end{align*}

For evaluation of the eigenvalue, we choose the linear interpolant $\phi_{h,\tau}$ in time together with an exact quadrature, i.e. $\bar\phi=\phi_{h,\tau}$. However, other choices, e.g., evaluation at the new or old time step or at the midpoint are also possible. This results to evaluations of the eigenvalue for $\phi_h^n,\phi_h^{n+1/2},\phi_h^{n+1}$ at every time step. Following \cite{Bartelsbook} we approximate the eigenvalue and eigenfunction pair $(-\lambda_h,w_h)$ in the finite element space via
\begin{align*}
    -\lambda_h(\gamma,\phi_h^*)\la w_h,\psi_h\ra = \tfrac{1}{2}\la \na w_h,\na \psi_h \ra + \frac{1}{\gamma}\la f'(\phi_{h}^*)w_h,\psi_h \ra \qquad \forall \psi_h\in\Vh,
\end{align*}
where $\phi_h^*$ is given by $\phi_h^n,\phi_h^{n+1/2}$ or $\phi_h^{n+1}$.

Under some generic assumptions, mainly $h\leq c\sqrt{\gamma}$ the following error bound was shown in  \cite[Proposition 6.7]{Bartelsbook}.
\begin{align}
   0 \leq \lambda(\gamma,\phi_h^*) - \lambda_h(\gamma,\phi_h^*) \leq c\frac{h^2}{\gamma^2}.
\end{align}

 Consequently, all the appearing quantities in the residual estimates are computable up to higher order terms.

\begin{remark}
In contrast to \cite{BMO2011} the $H^{-1}$-norm of the residuals plays a major role in the estimator, hence we use at least quadratic finite elements to gain a better scaling in the space discretisation parameter $h$. In the case of lowest-order piece-wise-linear elements, the estimates of the residuals in the negative norm behave similarly as in the $L^2$-norm.   
\end{remark}

\section{Numerical experiments}\label{sec:sim}

In this section, we illustrate the scaling of the residuals with respect to $h,\tau,\gamma$. Note that the residuals $r_1,r_2$ together with the eigenvalue $\lambda$ are the key quantities in the stability estimate. 

We fix the polynomial degree $p=2$, i.e. continuous piecewise quadratic elements. The nonlinear system is solved using a Newton scheme, while the linear systems are solved using a direct solver, here LU decomposition. The finite element scheme\footnote{The code can be found at \url{https://github.com/AaronBrunk1/Aposteriori_Allen_Cahn_Mobility}} is implemented in FreeFem++ \cite{Freefem}.

The following test case was proposed in  \cite{BMO2011}.
\begin{experiment}\label{exp:1}
Let $\Omega:=(-2,2)$, $T=0.1$, set $R_1:=4/10$ and $R_2=1$
with the distance functions $d_i(x):=|x|-R_j$. For given $\gamma>0$ we define
\begin{equation*}
    \phi_0(x):=-\tanh(d(x)/\sqrt{2\gamma}), \qquad d(x):=\max(-d_1(x),d_2(x)).
\end{equation*}
with the mixing potential $f(\phi)=\tfrac{1}{4}(\phi^2-1)^2$ and mobility function $b(\phi)=1 + \tfrac{1}{10}(\phi^2-1)^2.$
\end{experiment}

In Figure \ref{fig:snaps} we depict several snapshots of the time evolution of $\phi$. We can observe that the inner radius of the initial annulus shrinks over time until a topological singularity, here the collapse to a sphere, occurs around $t\approx 0.072$. From the experiment in \cite{BMO2011} we know that at a later time, a second topological singularity occurs when the sphere collapses to one point. After the second singularity, the solution is $\phi\equiv -1$ in the whole domain.
\begin{figure}[htbp!]
\centering
\footnotesize
\begin{tabular}{ccc}
\multicolumn{3}{c}{\hspace{-0.5em}\includegraphics[trim={10.0cm 30.4cm 10.0cm 3.0cm},clip,scale=0.33]{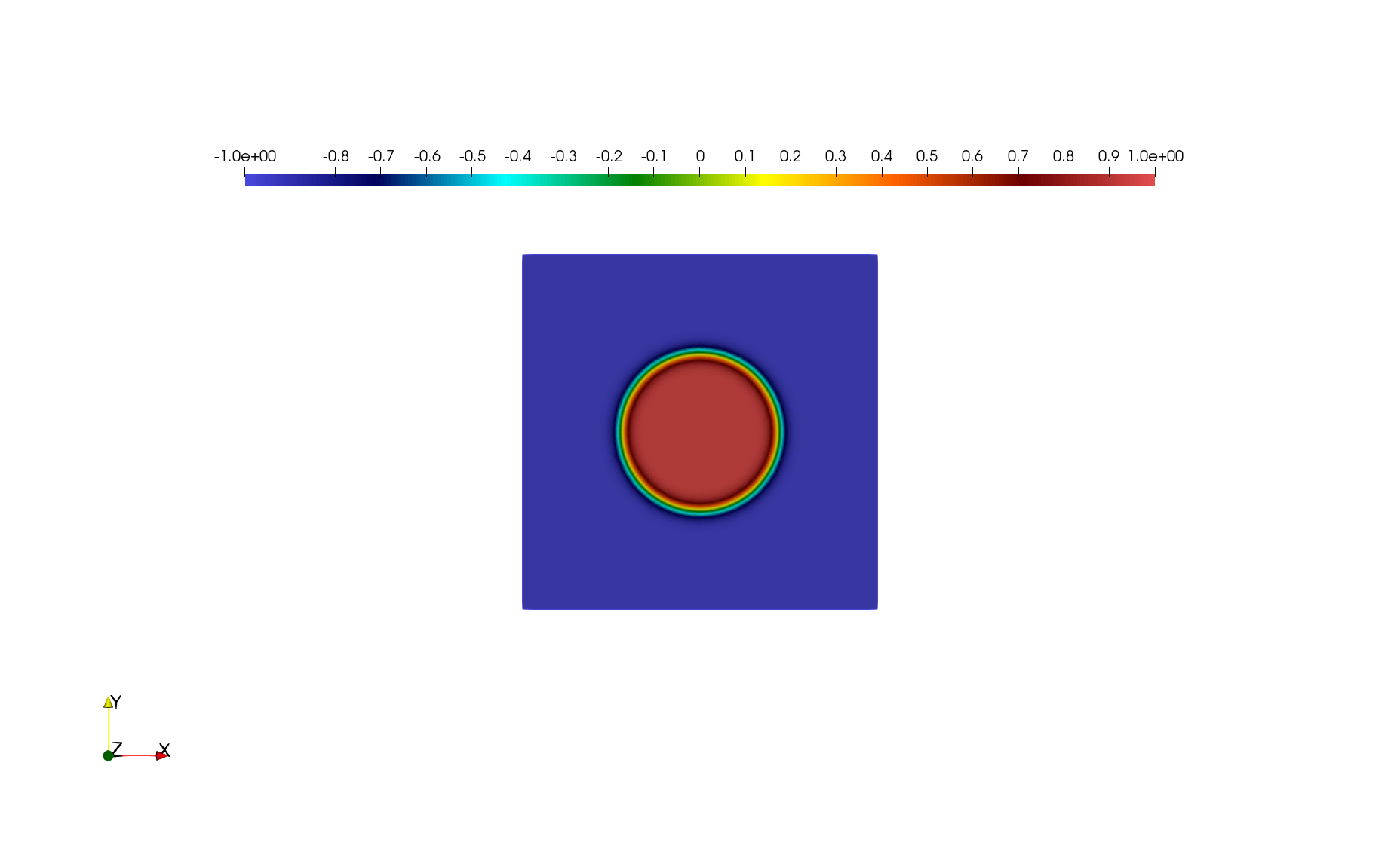}}\\
    \includegraphics[trim={15.3cm 1.4cm 19.5cm 5.0cm},clip,scale=0.158]{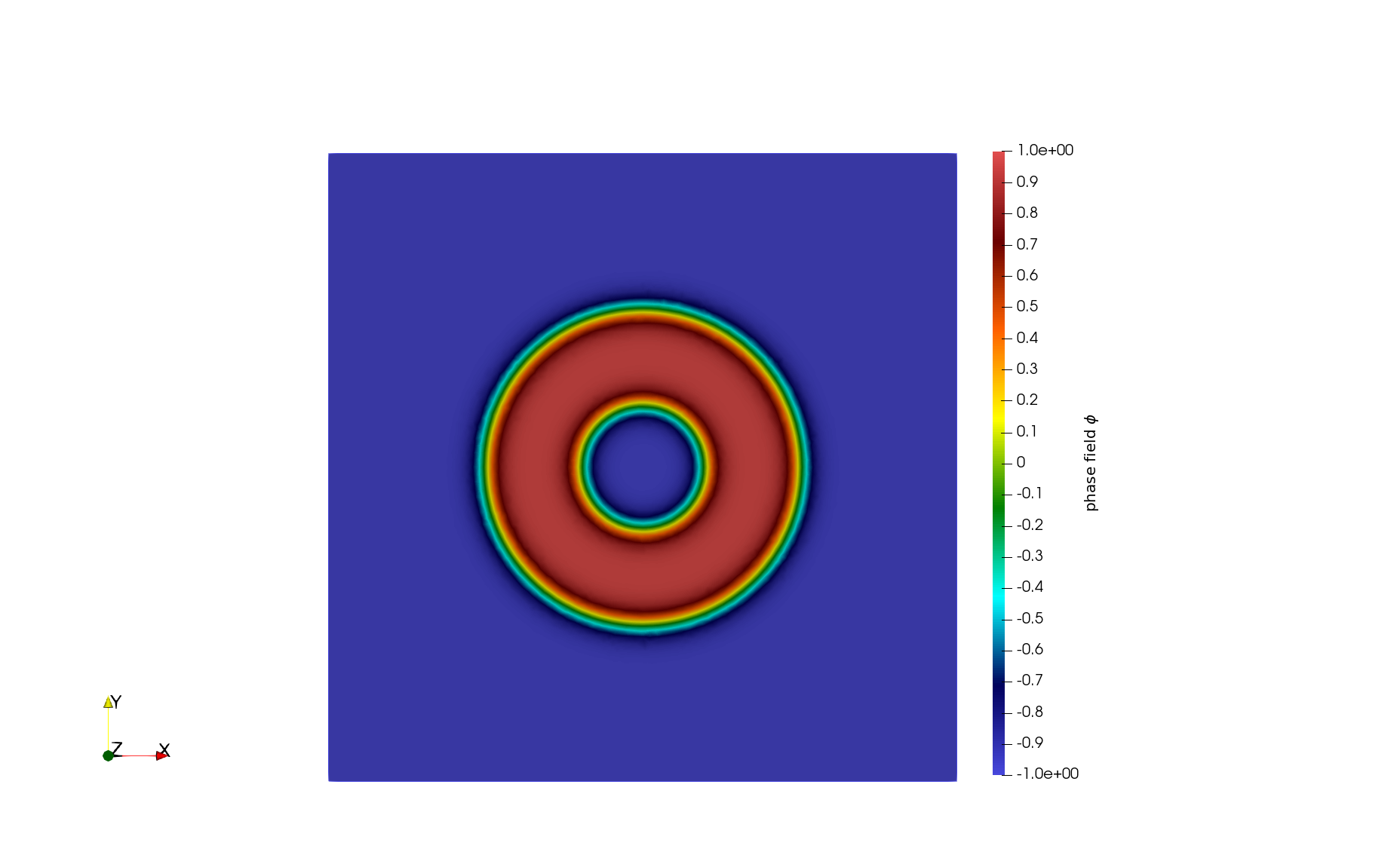} 
    &\hspace{-1em}
    \includegraphics[trim={15.3cm 1.4cm 19.5cm 5.0cm},clip,scale=0.158]{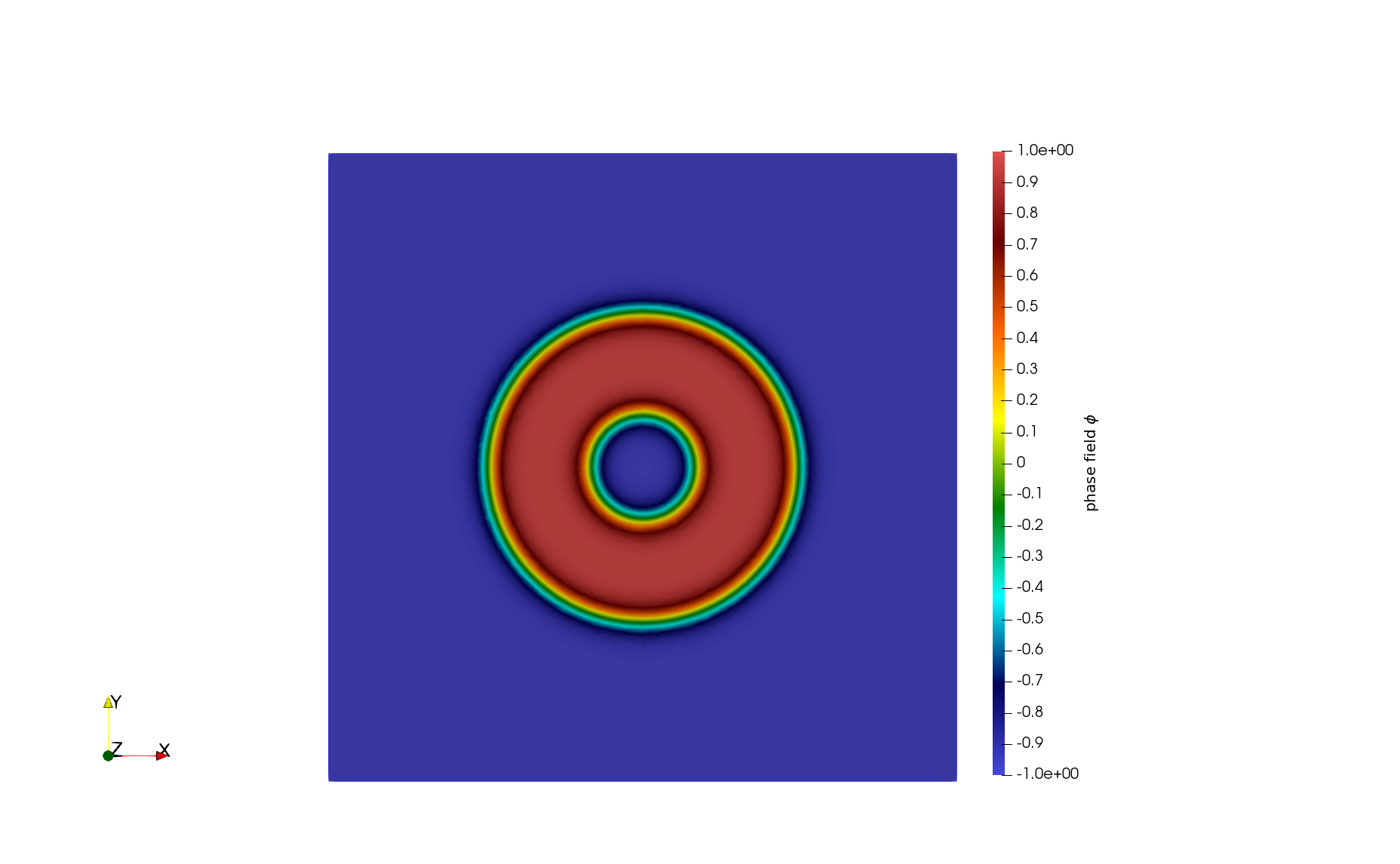}  
    &\hspace{-1em}
    \includegraphics[trim={15.3cm 1.4cm 19.5cm 5.0cm},clip,scale=0.158]{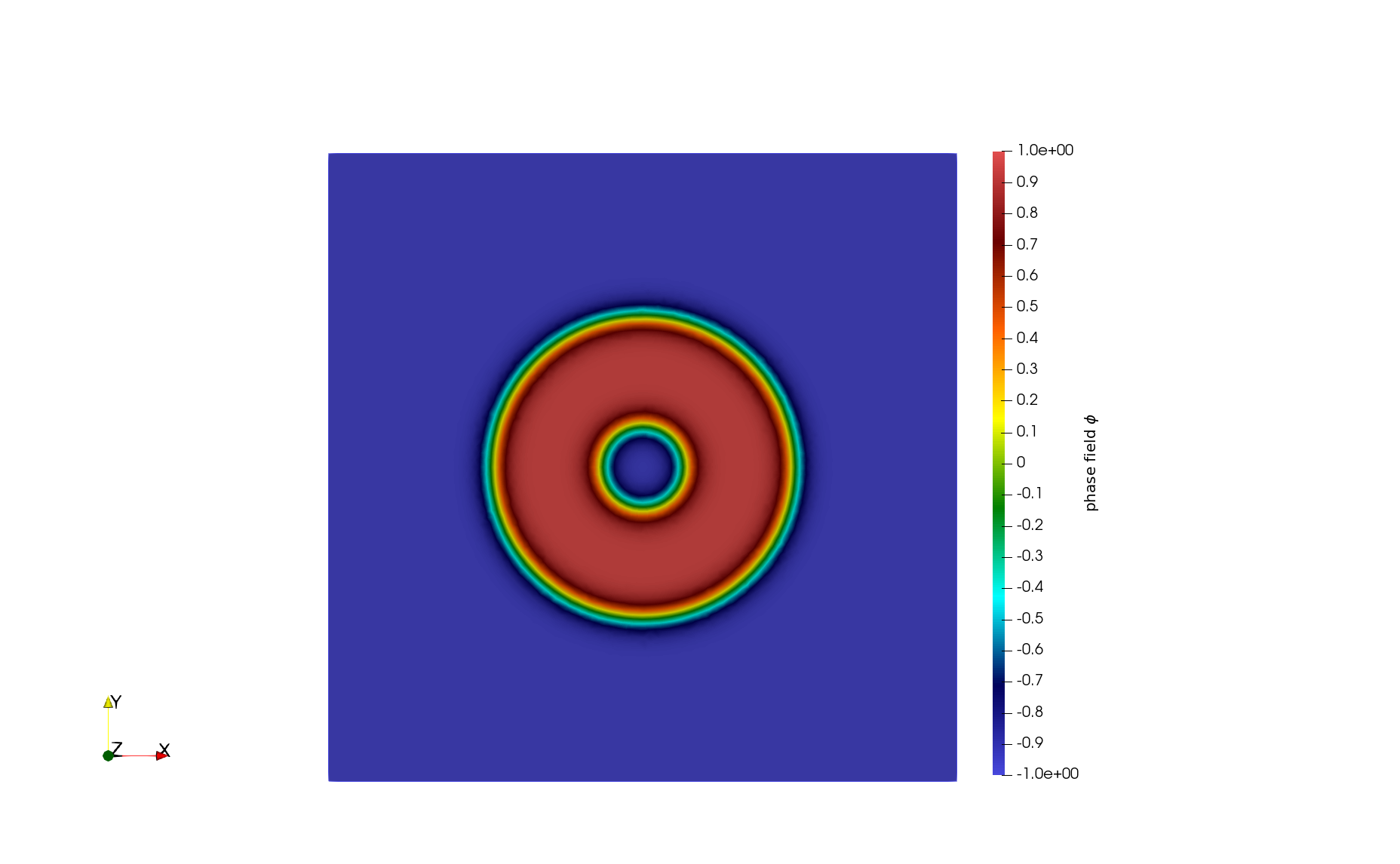} \\[-1.5em]
    \hspace{-1.5em}t=0 & \hspace{-1.5em}t=0.02 & \hspace{-1.5em}t=0.04 \\[1em]
    \includegraphics[trim={15.3cm 1.4cm 19.5cm 5.0cm},clip,scale=0.158]{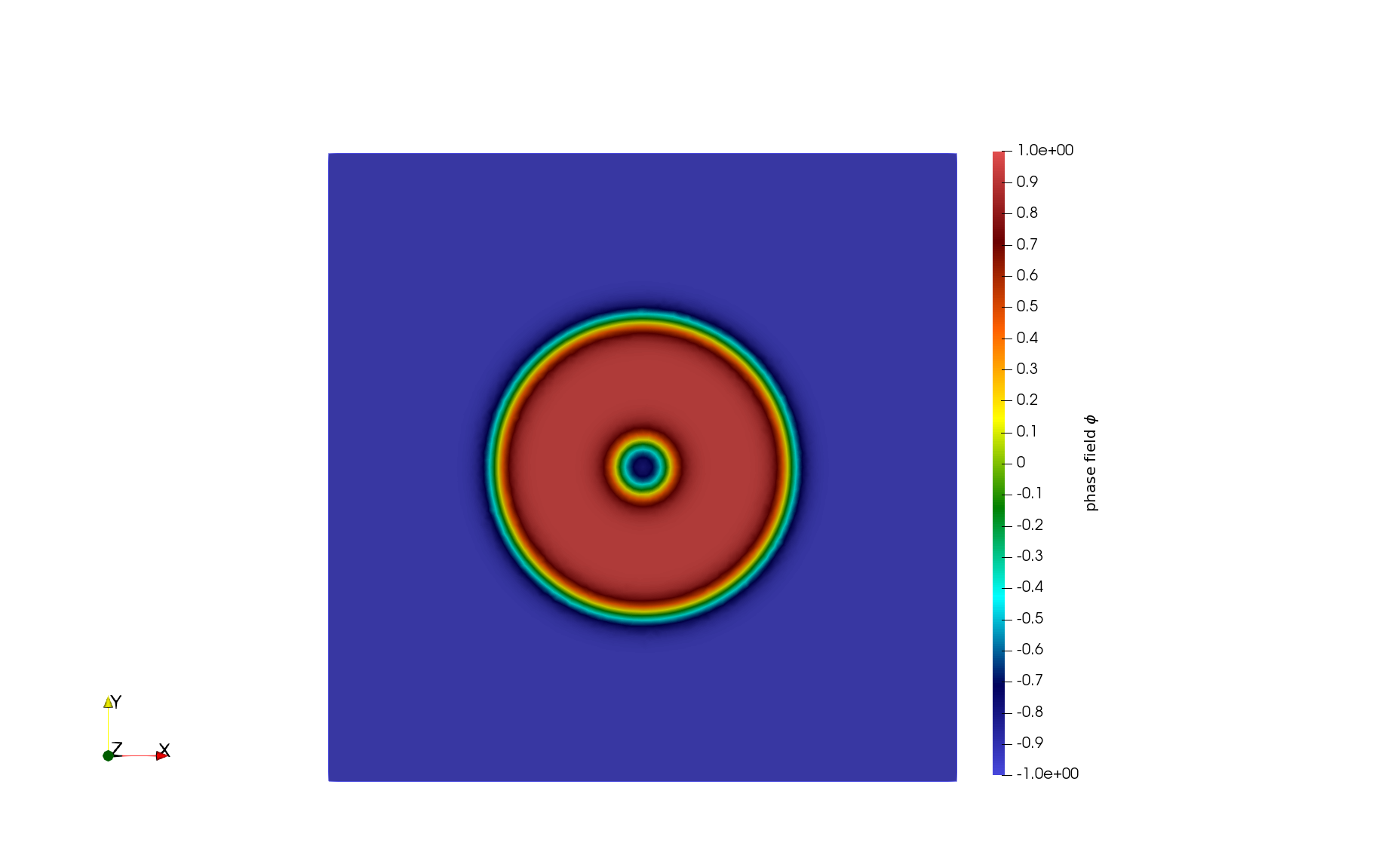}  
    &\hspace{-1em}
    \includegraphics[trim={15.3cm 1.4cm 19.5cm 5.0cm},clip,scale=0.158]{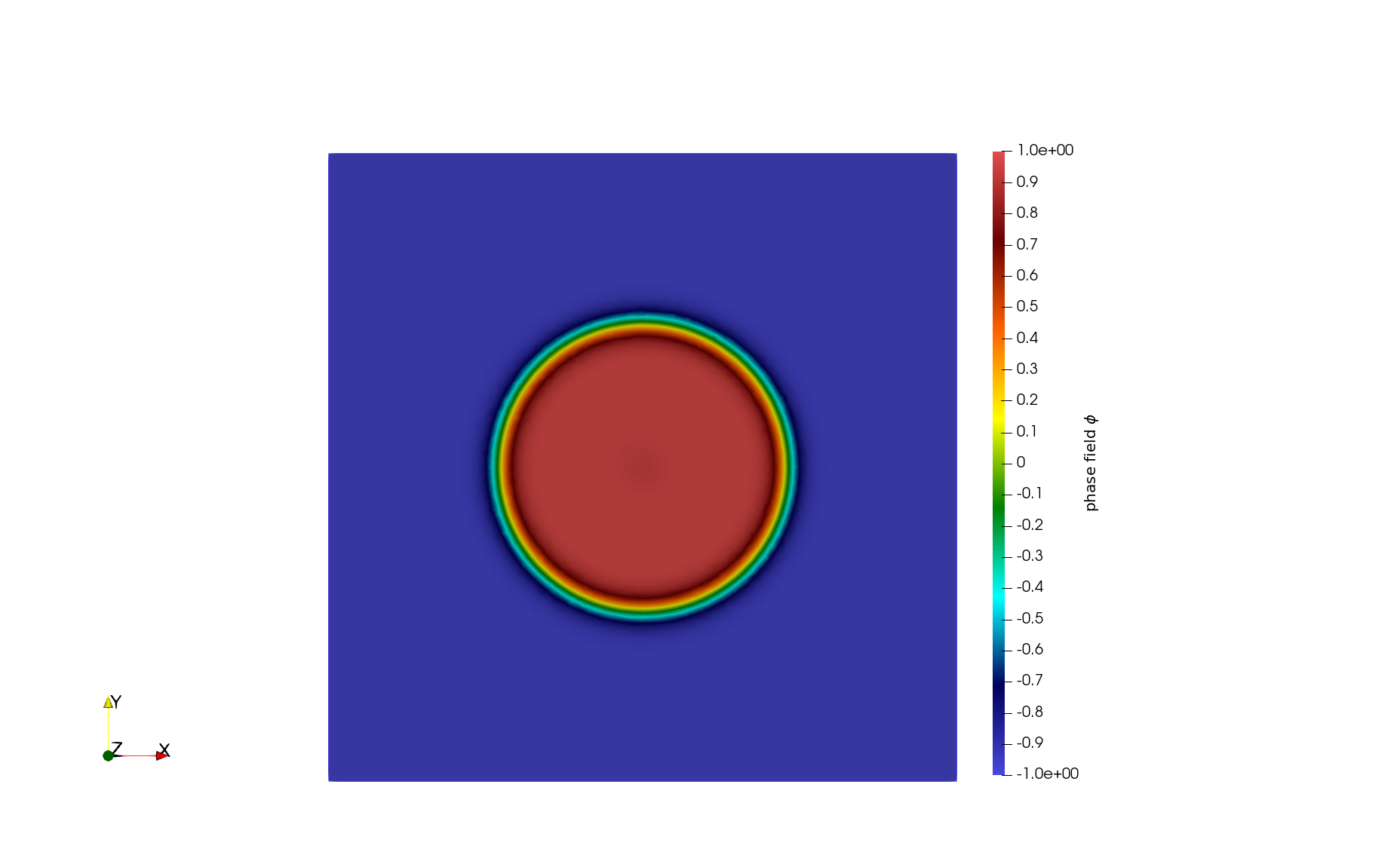} 
    &\hspace{-1em}
    \includegraphics[trim={15.3cm 1.4cm 19.5cm 5.0cm},clip,scale=0.158]{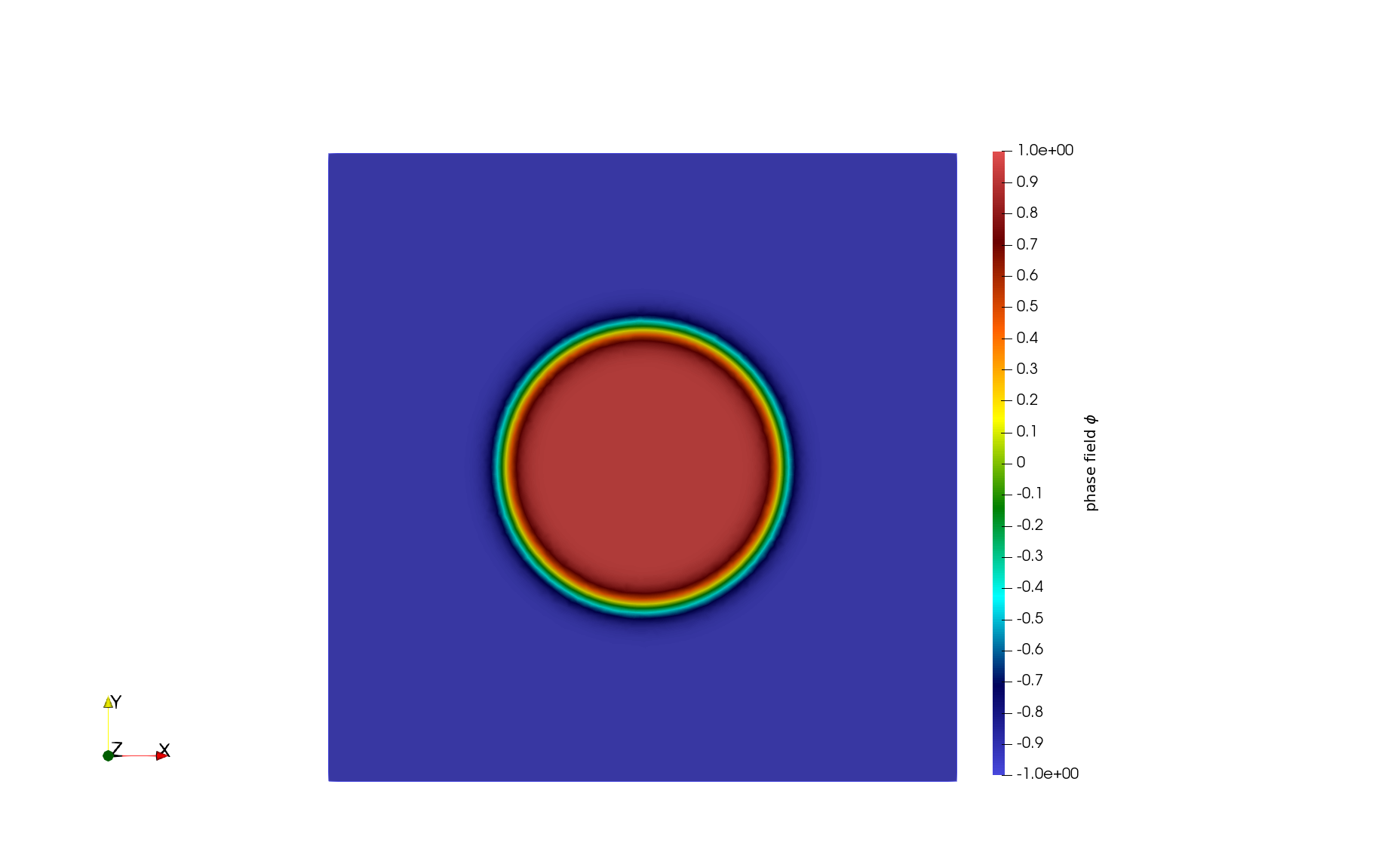} \\[-1.5em]
    \hspace{-1.5em}t=0.06 & \hspace{-1.5em}t=0.076 & \hspace{-1.5em}t=0.1     
\end{tabular}
    \caption{Time evolution of $\phi$ at different time steps. The topological singularity, here the collapse of the annulus to a sphere, happens around $t\approx 0.072$.}
    \label{fig:snaps}
\end{figure}

In the following we will study experimentally the scaling behaviour of the residuals $r_1,r_2$ in the $L^2(0,T;L^2(\Omega))$ and $L^2(0,T;H^{-1}(\Omega))$  norms with respect to $\gamma, h, \tau$ separately. Note that in all plots regarding the residual scaling, we do not plot norms of the residuals themselves but the computable estimators given in \eqref{eq:est1}--\eqref{eq:est4}. 
\begin{figure}
    \centering
    \includegraphics[scale=0.75]{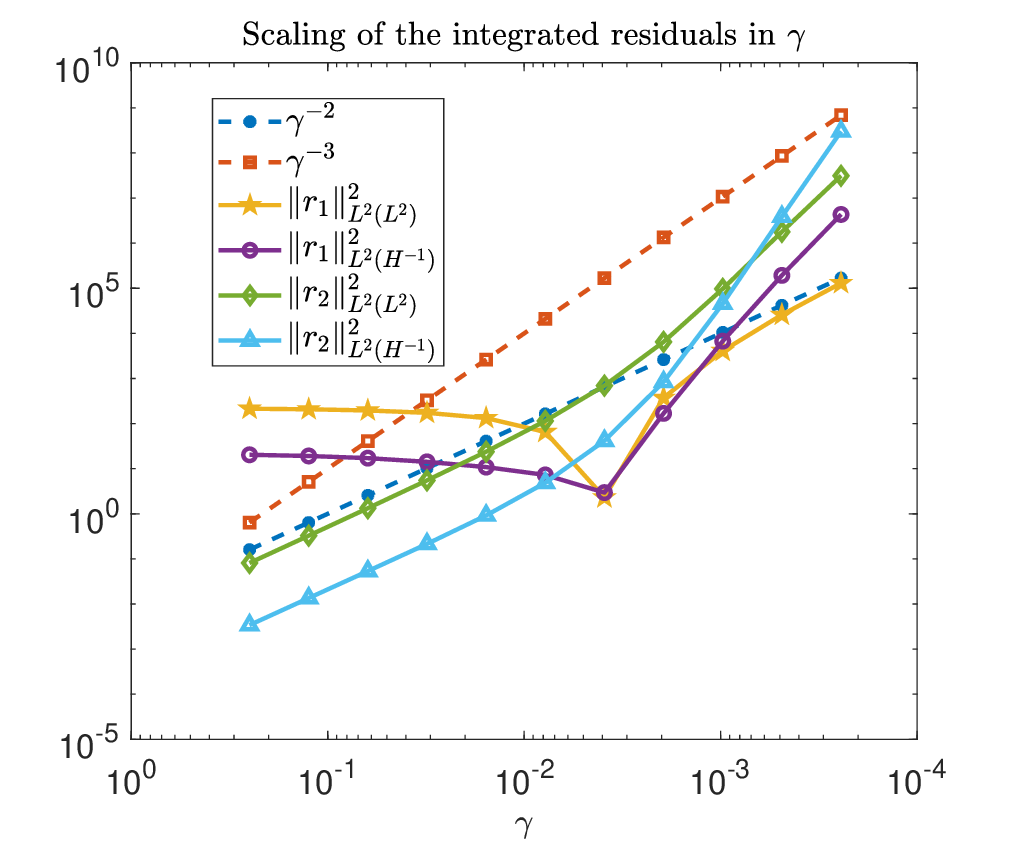}
    \caption{Experimental scaling of the residuals with respect to $\gamma$. Discretisation parameters $\gamma_k=2^{-k},h_{\min}\approx 0.03,h_{\max} \approx 0.065,\tau\approx 5\cdot 10^{-4},$ for $k=2,\ldots,11$.}
    \label{fig:gamma}
\end{figure}

\textbf{Scaling with respect to $\gamma$:}
In Figure \ref{fig:gamma} we have experimentally computed the scaling of the residuals with respect to the interface parameter $\gamma$. We can clearly observe that all residuals scale as expected with a negative power of $\gamma$. In more detail, it seems that $\norm{r_1}^2_{L^2(L^2)}$ scales as $\gamma^{-2}$, while all the other norms seem to scale even worse than $\gamma^{-3}$, but polynomial.
One can observe that the residuals in the $L^2(H^{-1})$-norm estimators scale worse in $\gamma$ than the $L^2(L^2)$-norm estimators. Recalling the estimators for the residuals \eqref{eq:est1}--\eqref{eq:est4} one can see that the negative norm estimators contain terms of the form $\norm{\nabla\hat\phi}_{0,3}^2$ which in a generic situation behaves as a polynomial in $\gamma^{-1}$. In this case one would expect  $\norm{\nabla\hat\phi}_{0,3}^2 \approx \gamma^{-2/3}$. Indeed in the plots the $L^2(H^{-1})$ norm is sometimes larger than the corresponding $L^2(L^2)$ norm. This stems from the fact that we plot the computable estimators and not the residuals, since for the residuals this would be impossible. Note that the $L^2(L^2)$ norms of the residuals appearing in Theorem \ref{thm:stab} are weighted by $\gamma^2$. Hence, the estimator is mainly influenced by the negative norms.

\begin{figure}
    \centering
    \includegraphics[scale=0.75]{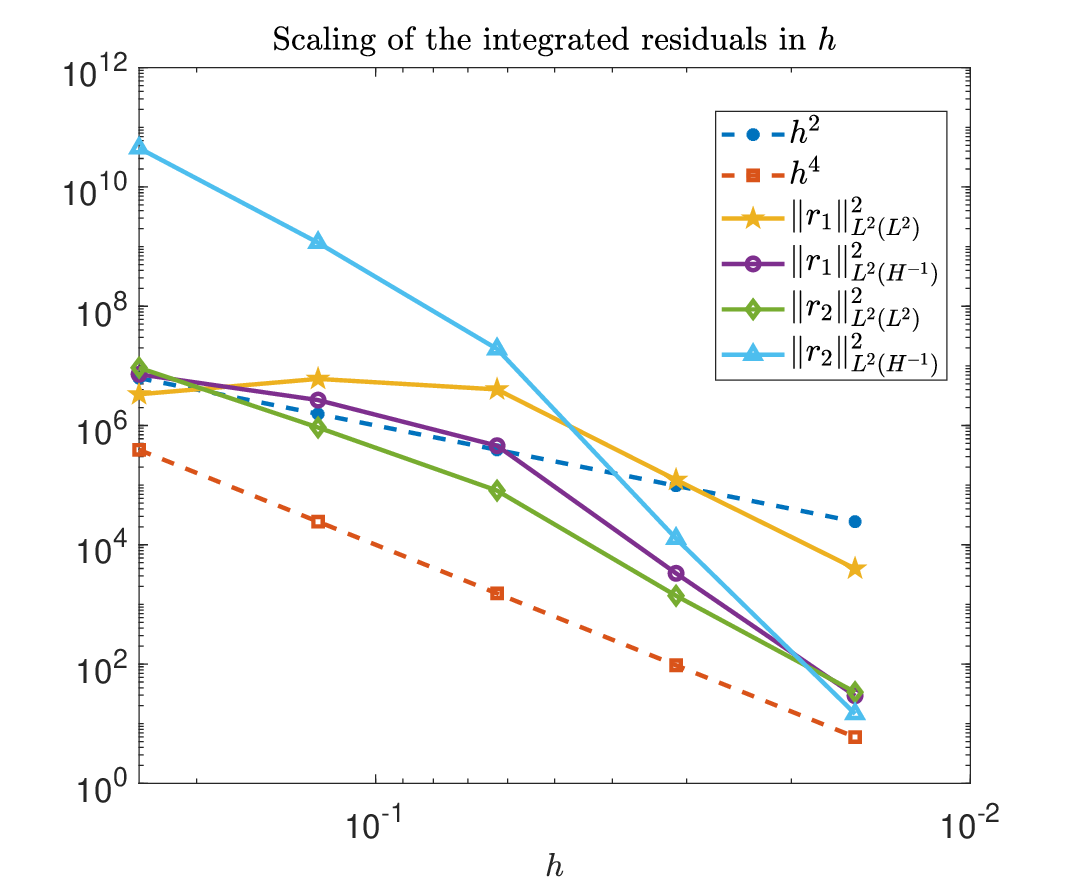}
    \caption{Experimental scaling of the residuals with respect to $h$. Discretisation parameters $\gamma=2^{-4},h_{k,\min}\approx 2^{-k},h_{k,\max}\approx 1.5\cdot 2^{-k} ,\tau=2.5\cdot 10^{-4}$ for $k=0,\ldots,4$.}
    \label{fig:space}
\end{figure}
\textbf{Scaling with respect to $h$:}
In Figure \ref{fig:space} we have experimentally computed the scaling of the residuals with respect to the space mesh-width $h$. We can clearly observe that all residuals scale with a positive power of $h$. In more details, it seems that $\norm{r_1}^2_{L^2(L^2)}$ and $\norm{r_2}^2_{L^2(L^2)}$ scale as $h^{4}$, while $\norm{r_1}^2_{L^2(H^{-1})}$ and $\norm{r_2}^2_{L^2(H^{-1})}$ scale even better.

\begin{figure}
    \centering
    \includegraphics[trim = 50 0 50 0, clip,scale=0.65]{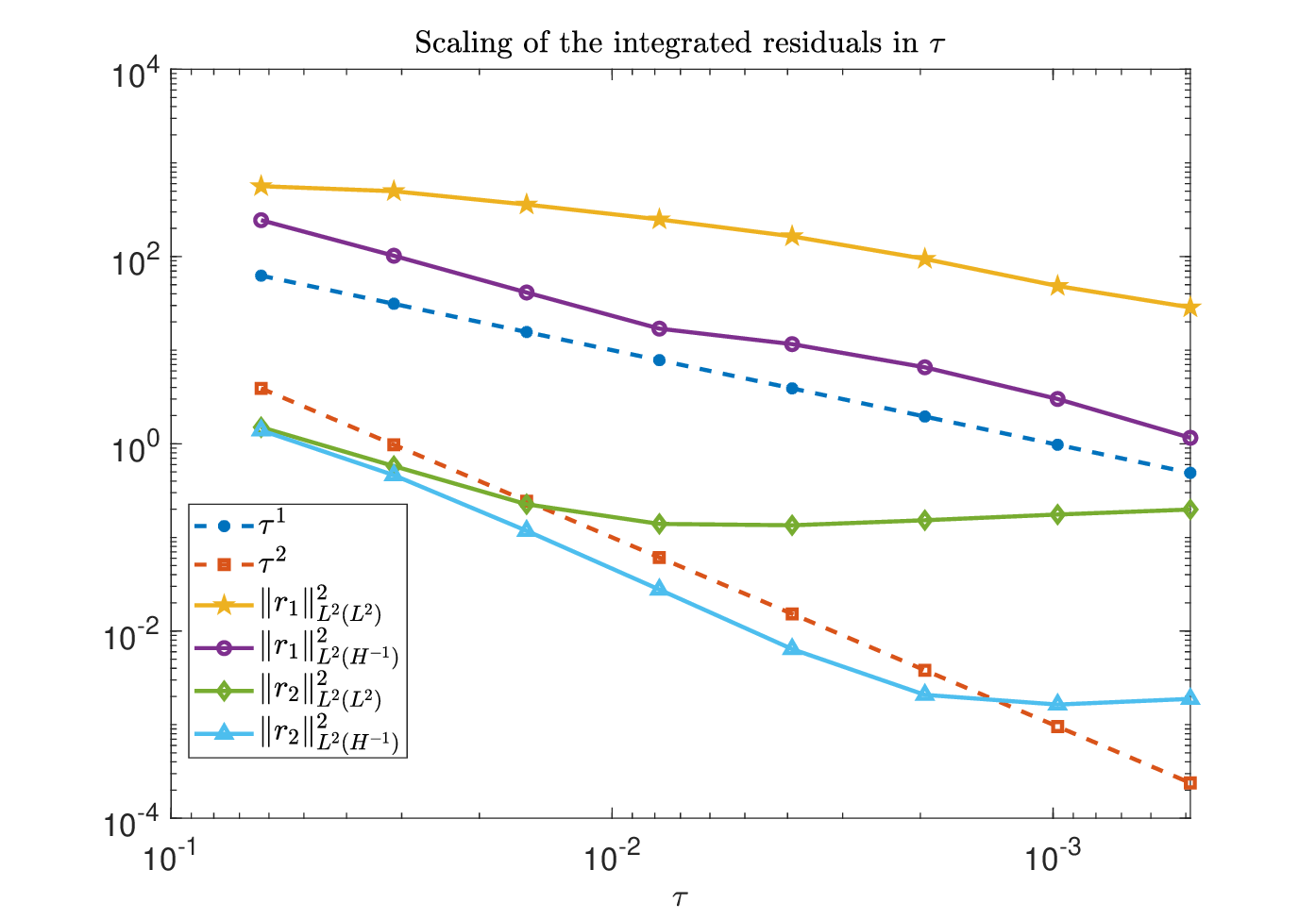}
    \caption{Experimental scaling of the residuals with respect to $h$. Discretisation parameters $\gamma=2^{-4},h_{\min}\approx 0.025,h_{\max} \approx 0.058,\tau_k=2^{-k}$ for $k=4,\ldots,11$.}
    \label{fig:time}
\end{figure}

\textbf{Scaling with respect to $\tau$:}
In Figure \ref{fig:time} we have experimentally computed the scaling of the residuals with respect to the time mesh-width $\tau$, until a certain error plateau is reached. We have confirmed that in this case the error is dominated by the spatial errors. In more details, it seems that $\norm{r_1}^2_{L^2(L^2)}$ and $\norm{r_1}^2_{L^2(H^{-1})}$ scale as $\tau$, while $\norm{r_1}^2_{L^2(L^2)}$ and $\norm{r_2}^2_{L^2(H^{-1})}$ scale as $\tau^2$ until the plateau is reached. The different behaviour between the residuals is due to the fact that the first residual contains a time derivative, while the second residual does not.

In the next part, we consider the scaling of the individual components of $M$ with respect to $\gamma$, cf. Theorem \ref{thm:stab}.

\textbf{Scaling of the eigenvalue with respect to $\gamma$:}
Finally, we consider the experimental scaling of the integrated eigenvalue $\lambda_+$ with respect to $\gamma$.
In Figure \ref{fig:eigenvalue} we can observe that the exponential of the integrated eigenvalue scales approximately as $\gamma^{-k}$, for $k\in[1,2]$. While the exponent is worse in comparison to the exponent obtained in \cite{BMO2011}, the logarithmic behaviour is still valid. In view of Remark \ref{rem:eig} this result experimentally justifies the use of the modified eigenvalue in our stability estimate. 

\begin{figure}
    \centering
    \includegraphics[scale=0.75]{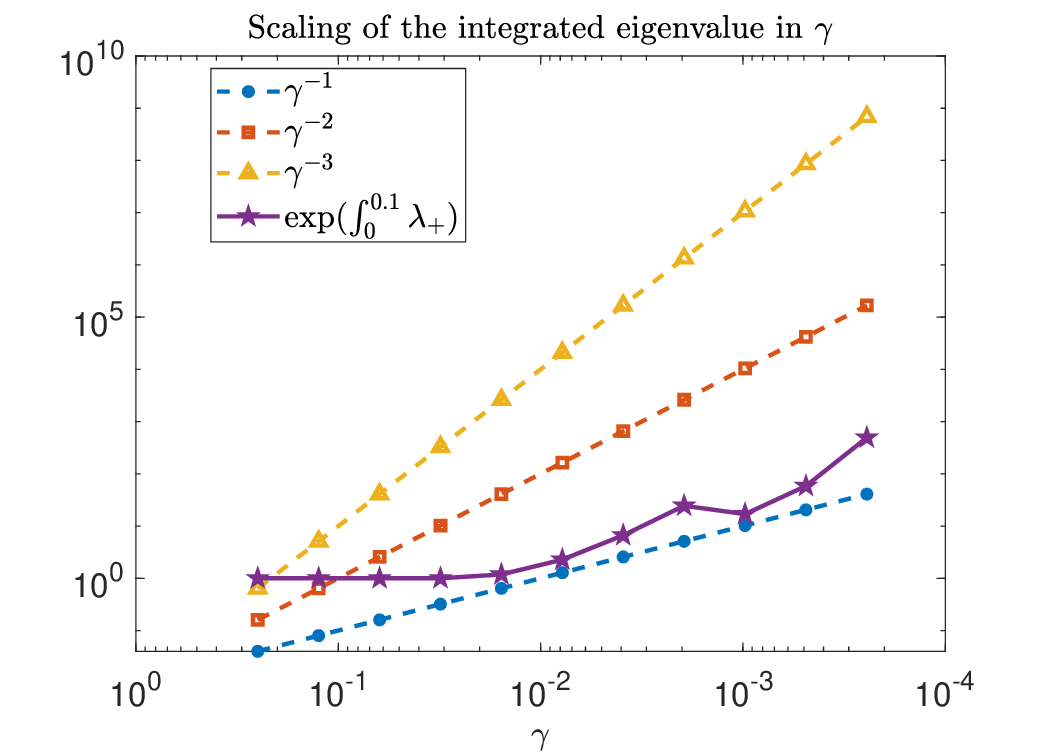}
    \caption{Experimental scaling of the integrated eigenvalue $\int_0^t \lambda_+$ with respect to $\gamma$. Discretisation parameters $\gamma_k=2^{-k},h_{\min}\approx 0.066,h_{\max} \approx 0.135,\tau=3\cdot 10^{-4}$ for $k=2,\ldots,12$.}
    \label{fig:eigenvalue}
\end{figure}

\textbf{Scaling of the remaining components of $M$ with respect to $\gamma$:}
\begin{figure}
    \centering
    \includegraphics[trim = 60 0 60 0, clip,scale=0.55]{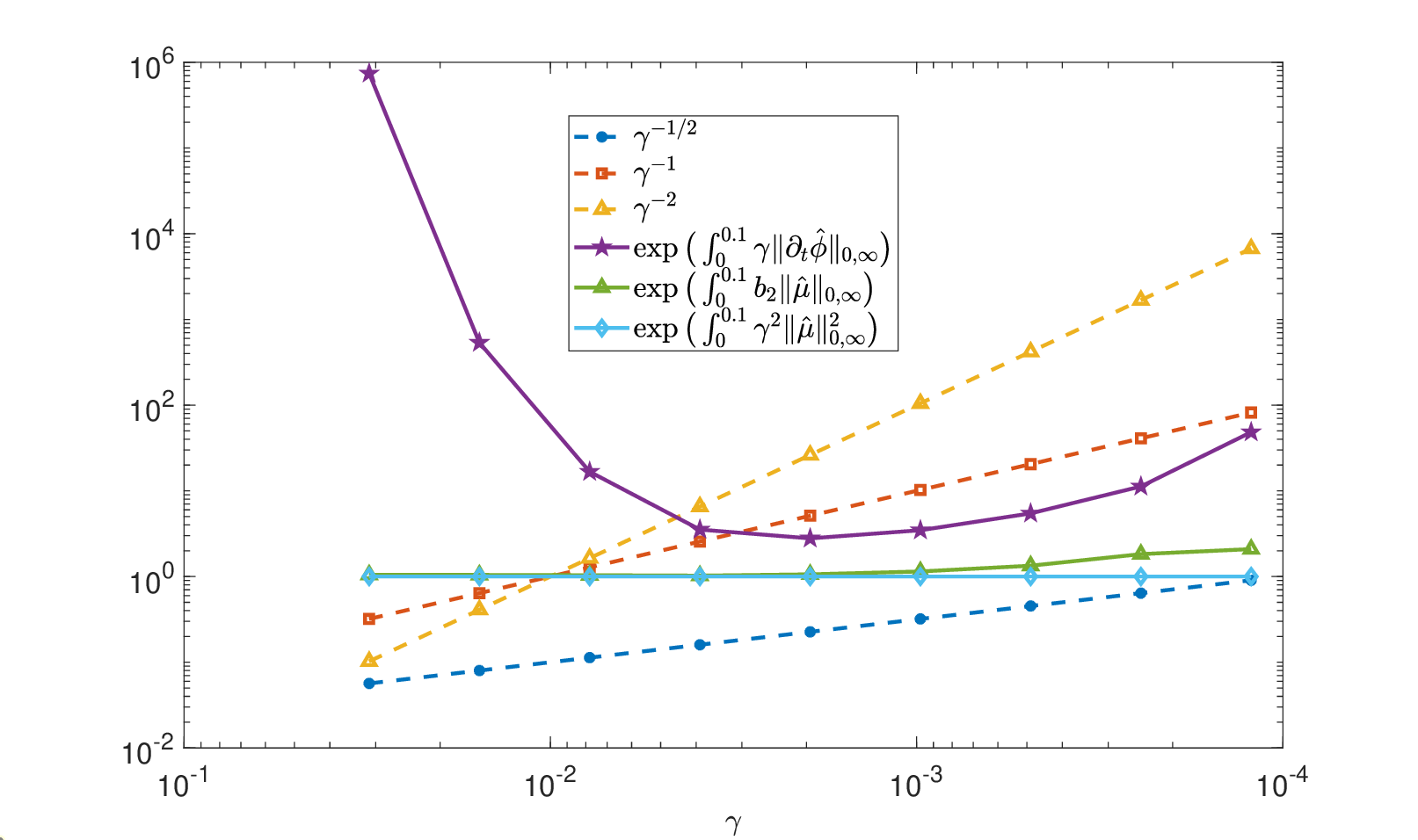}
    \caption{Experimental scaling of the remaining components of $M$, i.e.  $\int_0^t \gamma\norm{\dt\hat\phi}_{0,\infty}$, $\int_0^t b_2\norm{\hat\mu}_{0,\infty}$ and $\int_0^t \gamma^2\norm{\hat\mu}_{0,\infty}^2$  with respect to $\gamma$. Discretisation parameters $\gamma_k=2^{-k},h_{\min}\approx 0.066,h_{\max} \approx 0.135,\tau=2\cdot 10^{-4}$ for $k=5,\ldots,13$.}
    \label{fig:remainM}
\end{figure}

In Figure \ref{fig:remainM} we plot the behaviour of $\int_0^t \gamma\norm{\dt\hat\phi}_{0,\infty}, \int_0^t b_2\norm{\hat\mu}_{0,\infty}$ and $\int_0^t \gamma^2\norm{\hat\mu}_{0,\infty}^2$. The first term is estimated by similar residual estimators as in \eqref{eq:resinf}. While the second term is simply the upper bound estimating $b(\hat\phi)$ by $b_2$. We observe that all quantities scale like a polynomial in $\gamma^{-1}$ similarly as the integrated eigenvalue. Furthermore, we can see that $\int_0^t \gamma\norm{\dt\hat\phi}_{0,\infty}$ is even decreasing until a $\gamma$-threshold. A similar behaviour can be seen in the first residual, cf. Figure \ref{fig:gamma}. We can see that the third term is almost independent of $\gamma$.

\section{Conclusion}
In this work, we have derived and experimentally tested an a posteriori error estimator for the Allen-Cahn equation with variable non-degenerate mobility. The estimator is based on a conditional stability estimate which is derived using a weighted sum of two Bregman distances, based on the energy of the system and a suitable functional related to the mobility. We have experimentally confirmed that the main ingredients, the residuals and the integrated principal eigenvalue, of the error estimator scale as expected with a negative power of the interface parameter and positive power of the space and time discretisation parameters. Furthermore, we provided experimental evidence that the eigenvalue we employ here has a similar scaling as the eigenvalue that was used for constant mobility in \cite{BMO2011}. The rigorous analysis for the residuals scaling and the eigenvalue behaviour is postponed to future work.

\section*{Acknowledgments}
A.B. is grateful for financial support by the Deutsche Forschungsgemeinschaft (DFG, German Research Foundation) for the projects "Variational quantitative phase-field modeling and simulation of powder bed fusion additive manufacturing" 441153493 within SPP 2256 Variational Methods for Predicting Complex Phenomena in Engineering Structures and Materials and TRR146: Mutliscale simulation methods for soft matter systems. \\
J.G. is grateful for financial support by the Deutsche Forschungsgemeinschaft project “Dissipative solutions for the Navier-Stokes-Korteweg system and their numerical treatment” 525866748 within SPP 2410 Hyperbolic Balance Laws in Fluid Mechanics: Complexity, Scales, Randomness (CoScaRa). \\
M.L. has been funded by the Deutsche Forschungsgemeinschaft within
TRR 146 (Project Number 233630050) as well as within SPP 2410 (Project Number 525800857). She is grateful to the Gutenberg Research College and
Mainz Institute of Multiscale Modelling for supporting her research.

 \bibliographystyle{abbrv}
 \bibliography{mob}

\begin{thebibliography}{10}

\bibitem{Akrivis2020}
G.~Akrivis and B.~Li.
\newblock Error estimates for fully discrete {BDF} finite element
  approximations of the {A}llen–{C}ahn equation.
\newblock {\em IMA J. Numer.}, 42(1):363–391, 2020.

\bibitem{ALLEN1979}
S.~M. Allen and J.~W. Cahn.
\newblock A microscopic theory for antiphase boundary motion and its
  application to antiphase domain coarsening.
\newblock {\em Acta Metall}, 27(6):1085--1095, 1979.

\bibitem{Araya2006}
R.~Araya, E.~Behrens, and R.~Rodríguez.
\newblock A posteriori error estimates for elliptic problems with {D}irac delta
  source terms.
\newblock {\em Numer. Math.}, 105(2):193–216, 2006.

\bibitem{Barrett2002}
J.~W. Barrett.
\newblock {Finite element approximation of an Allen-Cahn/Cahn-Hilliard system}.
\newblock {\em IMA J. Numer.}, 22(1):11–71, 2002.

\bibitem{Bartelsbook}
S.~Bartels.
\newblock {\em Numerical Methods for Nonlinear Partial Differential Equations}.
\newblock Springer International Publishing, 2015.

\bibitem{Bartels2015}
S.~Bartels.
\newblock Robustness of error estimates for phase field models at a class of
  topological changes.
\newblock {\em Comput. Methods Appl. Mech. Engrg.}, 288:75--82, 2015.

\bibitem{Bartels2011b}
S.~Bartels and R.~M\"{u}ller.
\newblock Error control for the approximation of {A}llen–{C}ahn and
  {C}ahn–{H}illiard equations with a logarithmic potential.
\newblock {\em Numer. Math.}, 119(3):409–435, 2011.

\bibitem{Bartels2011a}
S.~Bartels and R.~M\"{u}ller.
\newblock Quasi-optimal and robust a posteriori error estimates in
  $l^{\infty}(l^{2})$ for the approximation of {A}llen-{C}ahn equations past
  singularities.
\newblock {\em Math. Comput.}, 80(274):761–761, 2011.

\bibitem{BMO2011}
S.~Bartels, R.~M\"{u}ller, and C.~Ortner.
\newblock Robust a priori and a posteriori error analysis for the approximation
  of {A}llen-{C}ahn and {G}inzburg-{L}andau equations past topological changes.
\newblock {\em SIAM J. Numer. Anal.}, 49(1):110--134, 2011.

\bibitem{BK2019}
D.~Bl\"{o}mker and M.~Kamrani.
\newblock Numerically computable a posteriori-bounds for the stochastic
  {A}llen-{C}ahn equation.
\newblock {\em BIT}, 59(3):647--673, 2019.

\bibitem{Chen1994}
X.~Chen.
\newblock {Spectrum for the Allen-Cahn, Cahn-Hillard, and phase-field equations
  for generic interfaces}.
\newblock {\em Commun. Partial Differ. Equ.}, 19(7-8):1371--1395, 1994.

\bibitem{CHY2020}
Y.~Chen, Y.~Huang, and N.~Yi.
\newblock Recovery type a posteriori error estimation of adaptive finite
  element method for {A}llen–{C}ahn equation.
\newblock {\em J. Comput. Appl.}, 369:112574, 2020.

\bibitem{Chry2020}
K.~Chrysafinos, E.~H. Georgoulis, and D.~Plaka.
\newblock A posteriori error estimates for the {A}llen--{C}ahn problem.
\newblock {\em SIAM J. Numer. Anal.}, 58(5):2662--2683, 2020.

\bibitem{Makridakis09}
A.~Demlow, O.~Lakkis, and C.~Makridakis.
\newblock A posteriori error estimates in the maximum norm for parabolic
  problems.
\newblock {\em SIAM J. Numer. Anal.}, 47(3):2157--2176, 2009.

\bibitem{EG1996}
C.~M. Elliott and H.~Garcke.
\newblock On the {C}ahn–{H}illiard equation with degenerate mobility.
\newblock {\em SIAM J. Math. Anal.}, 27(2):404--423, 1996.

\bibitem{FP2003}
X.~Feng and A.~Prohl.
\newblock Numerical analysis of the {A}llen-{C}ahn equation and approximation
  for mean curvature flows.
\newblock {\em Numer. Math.}, 94(1):33--65, 2003.

\bibitem{Feng2005}
X.~Feng and H.-J. Wu.
\newblock A posteriori error estimates and an adaptive finite element method
  for the {A}llen–{C}ahn equation and the mean curvature flow.
\newblock {\em J. Sci. Comput.}, 24(2):121–146, 2005.

\bibitem{GM2014}
E.~H. Georgoulis and C.~Makridakis.
\newblock On a posteriori error control for the {A}llen-{C}ahn problem.
\newblock {\em Math. Methods Appl. Sci.}, 37(2):173--179, 2014.

\bibitem{GLW2017}
Z.~Guan, J.~Lowengrub, and C.~Wang.
\newblock Convergence analysis for second-order accurate schemes for the
  periodic nonlocal {A}llen-{C}ahn and {C}ahn-{H}illiard equations.
\newblock {\em Math. Methods Appl. Sci.}, 40(18):6836--6863, 2017.

\bibitem{Freefem}
F.~Hecht.
\newblock New development in {F}reefem++.
\newblock {\em J. Numer. Math.}, 20(3-4):251--265, 2012.

\bibitem{Heida2015}
M.~Heida.
\newblock On systems of {C}ahn-{H}illiard and {A}llen-{C}ahn equations
  considered as gradient flows in {H}ilbert spaces.
\newblock {\em J. Math. Anal. Appl.}, 423(1):410--455, 2015.

\bibitem{Huang2019}
Y.~Huang, W.~Yang, H.~Wang, and J.~Cui.
\newblock Adaptive operator splitting finite element method for
  {A}llen–{C}ahn equation.
\newblock {\em Numer. Methods Partial Differ. Equ.}, 35(3):1290–1300, 2019.

\bibitem{Kessler2004}
D.~Kessler, R.~H. Nochetto, and A.~Schmidt.
\newblock A posteriori error control for the {A}llen–{C}ahn problem:
  circumventing {G}ronwall’s inequality.
\newblock {\em ESIAM Math. Model Numer.}, 38(1):129–142, 2004.

\bibitem{Li2019}
C.~Li, Y.~Huang, and N.~Yi.
\newblock An unconditionally energy stable second order finite element method
  for solving the {A}llen–{C}ahn equation.
\newblock {\em J. Comput. Appl.}, 353:38–48, 2019.

\bibitem{MN2003}
C.~Makridakis and R.~H. Nochetto.
\newblock Elliptic reconstruction and a posteriori error estimates for
  parabolic problems.
\newblock {\em SIAM J. Numer. Anal.}, 41(4):1585--1594, 2003.

\bibitem{Schatzman1995}
P.~D. Mottoni and M.~Schatzman.
\newblock {Geometrical Evolution of Developed Interfaces}.
\newblock {\em Trans. Am. Math. Soc.}, 347(5):1533--1589, 1995.

\bibitem{Shen2016}
J.~Shen, T.~Tang, and J.~Yang.
\newblock On the maximum principle preserving schemes for the generalized
  {A}llen–{C}ahn equation.
\newblock {\em Commun. Math. Sci.}, 14(6):1517–1534, 2016.

\bibitem{Shen2010}
J.~Shen and X.~Yang.
\newblock Numerical approximations of {A}llen-{C}ahn and {C}ahn-{H}illiard
  equations.
\newblock {\em Discrete Contin. Dyn. Syst. Ser. A}, 28(4):1669--1691, 2010.

\bibitem{Xiao2022}
X.~Xiao and X.~Feng.
\newblock A second-order maximum bound principle preserving operator splitting
  method for the {A}llen–{C}ahn equation with applications in multi-phase
  systems.
\newblock {\em Math. Comput. Simul.}, 202:36–58, 2022.

\bibitem{Yang2009}
X.~Yang.
\newblock Error analysis of stabilized semi-implicit method of {A}llen-{C}ahn
  equation.
\newblock {\em Discrete Contin. Dyn. Syst. Ser. B}, 11(4):1057–1070, 2009.

\end{thebibliography}

\end{document}